\documentclass[11pt,a4paper]{article}
\usepackage{amsmath,amsfonts,amssymb,amsthm}
\title{Local tensor valuations}
\renewcommand{\thefootnote}{\fnsymbol{footnote}}
\author{Daniel Hug and Rolf Schneider}
\date{}
\newcommand{\R}{{\mathbb R}}
\newcommand{\Sn}{{\mathbb S}^{n-1}}
\newcommand{\Kn}{{\mathcal K}^n}
\newcommand{\Pn}{{\mathcal P}^n}
\newcommand{\F}{{\mathcal F}}
\newcommand{\B}{{\mathcal B}}
\newcommand{\N}{{\mathbb N}}
\newcommand{\T}{{\mathbb T}}
\newcommand{\Ha}{{\mathcal H}}
\newcommand{\D}{{\rm d}}
\newcommand{\fed}{\,\rule{.1mm}{.20cm}\rule{.20cm}{.1mm}\,}

\newcommand{\RR}{\ensuremath{\mathbb{R}}}

\DeclareMathOperator{\Nor}{Nor}
\DeclareMathOperator{\sgn}{sgn}

\newtheorem{theorem}{Theorem}
\newtheorem{lemma}{Lemma}[section]

\newtheorem{corollary}{Corollary}[section]
\newtheorem{remark} {Remark}[section]

\catcode`@=11
\def\section{
\setcounter{equation}{0} \setcounter{theorem}{0} \@startsection
{section}{1}{\z@}{-4.0ex plus -1ex minus
    -.2ex}{2.3ex plus .2ex}{\bf\Large}}
\def\subsection{\@startsection{subsection}{2}{\z@}{-3.25ex plus-1ex
    minus-.2ex}{1.5ex plus.2ex}{\reset@font\bf\LARGE}}

\begin{document}

\maketitle

\begin{abstract}
The local Minkowski tensors are valuations on the space of convex bodies in Euclidean space with values in a space of tensor measures. They generalize at the same time the intrinsic volumes, the curvature measures and the isometry covariant Minkowski tensors that were introduced by McMullen and characterized by Alesker. In analogy to the characterization theorems of Hadwiger and Alesker, we give here a complete classification of all locally defined tensor measures on convex bodies that share with the local Minkowski tensors the basic geometric properties of isometry covariance and weak continuity.\\[1mm]
{\em 2010 Mathematics Subject Classification:} primary 52A20, secondary 52A22\\[1mm]
{\em Keywords:} valuation, Minkowski tensor, tensor valuation, support measure, characterization theorem, weak continuity, normal cycle
\end{abstract}

\renewcommand{\thefootnote}{{}}
\footnote{We acknowledge the support of the German Research Foundation (DFG) through the Research Unit `Geometry and Physics of Spatial Random
Systems' under the grant HU 1874/2-1.}

\section{Introduction}\label{sec1}

Additivity of set functions is a basic notion of great use in different parts of mathematics. While in its stronger version of countable additivity it constitutes the foundation of measure theory and thus is ubiquitous in analysis, a seemingly rudimentary form of additivity leads to a surprisingly rich theory in the geometry of sets in Euclidean and other spaces. Restricting ourselves here to the space $\Kn$ of convex bodies (nonempty, compact, convex subsets) in Euclidean space $\R^n$, we say that a function $\varphi$ on $\Kn$ with values in an abelian group (or an abelian semigroup with cancellation law) is {\em additive} or a {\em valuation} if
$$ \varphi(K\cup M)+\varphi(K\cap M)=\varphi(K)+\varphi(M)$$
whenever $K,M,K\cup M\in\Kn$. The space $\Kn$ may be replaced by an intersectional subclass, and also suitable classes of sets more general than convex bodies may be admitted. Real valuations on convex polytopes were first put to good use in Dehn's solution of Hilbert's third problem, and since then they play an essential role in the dissection theory of polytopes (see \cite{McM93}). 

An important development was initiated when Blaschke \cite[\S43]{Bla37} treated the kinematic integral
$$ \psi(K,M)= \int_{G_n}\chi(K\cap gM)\,\mu(\D g)$$
for $K,M\in\Kn$ (it is insignificant that Blaschke considered a slightly different situation). Here, $\chi(L)=1$ for $L\in\Kn$, $\chi(\emptyset)=0$, and $\mu$ denotes the suitably normalized Haar measure on $G_n$, the motion group of $\R^n$. Thus, $\psi(K,M)$ is the total invariant measure of the rigid motions that bring $M$ into a position of nonempty intersection with $K$. The determination of such integrals is a prerequisite for answering some classical questions on geometric probabilities. Blaschke pointed out that the function $\psi(K,\cdot)$, for fixed $K$, is a rigid motion invariant valuation on $\Kn$, and that, therefore, it would be helpful to have an axiomatic characterization of the classical examples of such valuations, the intrinsic volumes. These functionals on $\Kn$ can be derived from the notion of volume, via the coefficients in the Steiner formula
$$ V_n(K+\rho B^n) = \sum_{j=0}^n \rho^{n-j}\kappa_{n-j}V_j(K),\qquad \rho\ge 0,$$
for $K\in\Kn$. Here $V_n$ denotes the volume, $+$ is vector addition, $B^n$ is the unit ball, and $\kappa_n=V_n(B^n)$. The function $V_j:\Kn\to\R$ defined in this way for $j=0,\dots,n$, the $j$th {\em intrinsic volume}, is a rigid motion invariant valuation, but it shares these properties with many other functions on $\Kn$. To single out the intrinsic volumes, Blaschke originally suggested the property of local boundedness, and Hadwiger repeatedly (\cite[p. 346]{Had50}, \cite[footnote 3]{Had51}) emphasized that a characterization on this basis would be desirable and useful. However, an example in \cite[p. 229]{McMS83} shows that this is not possible. It is the condition of continuity, with respect to the Hausdorff metric, with which Hadwiger succeeded.

\vspace{2mm}

\noindent{\bf Hadwiger's characterization theorem} \;{\em If $\varphi:\Kn\to\R$ is a continuous valuation which is invariant under proper rigid motions, then there are constants $c_0,\dots,c_n$ such that
$$ \varphi(K)= \sum_{j=0}^n c_jV_j(K)$$
for $K\in\Kn$.}

\vspace{2mm}

Hadwiger proved this in \cite{Had51} for $n=3$ and in \cite{Had52} for general $n$ (the proof is reproduced in \cite[6.1.10]{Had57}) and presented integral-geometric applications in \cite{Had50} and \cite{Had56}. For example, since the function $\psi(K,\cdot)$ defined above is continuous, Hadwiger's theorem yields that it is of the form $\psi(K,M)= \sum_{j=0}^n c_j(K)V_j(M)$. A repetition of the argument with variable $K$ shows that $\psi(K,M)= \sum_{i,j=0}^n c_{ij}V_i(K)V_j(M)$, with constants $c_{ij}$, which are then easily determined by inserting balls of different radii. The result is the {\em principal kinematic formula} of Blaschke, Chern and Santal\'{o}, for convex bodies. In this elegant way, several integral-geometric formulas can be proved. Usually, they admit also other, though more complicated proofs. However, for the following result, called {\em Hadwiger's general integral-geometric theorem}, the approach via Hadwiger's characterization theorem (reproduced in \cite[Theorem 5.1.2]{SW08}) is up to now the only known proof. If $\varphi:\Kn\to\R$ is a continuous valuation, then
$$ \int_{G_n} \varphi(K\cap g M)\,\mu(\D g) =\sum_{j=0}^n\varphi_{n-j}(K) V_j(M)$$
for $K,M\in\Kn$, where the coefficients $\varphi_{n-j}(K)$ are given by
$$ \varphi_{n-j}(K) =\int_{A(n,j)} \varphi(K\cap E)\,\mu_j(\D E);$$
here $A(n,j)$ is the affine Grassmannian of $j$-planes in $\R^n$ and $\mu_j$ is its suitably normalized motion invariant measure.

In view of such applications and their interpretations, Hadwiger's characterization theorem is esteemed so highly that, for example, Gian--Carlo Rota \cite{Rot98}, in a Colloquium Lecture at an Annual Meeting of the American Mathematical Society, called it the `Main Theorem of Geometric Probability'. Nowadays, the use of real translation invariant valuations (called `geometric functionals', if they have certain boundedness and mesasurability properties) in stochastic geometry goes far beyond classical geometric probabilities. We refer to \cite[Chap. 9]{SW08}, where densities of geo\-metric functionals for random sets are treated, and for example to \cite{HLS13}, which investigates asymptotic covariances and multivariate central limit theorems for geometric functionals in relation to Boolean models.

The valuation property is shared by many functions arising naturally in convex geometry; they may, for example, be vector-valued, measure-valued, body-valued, or function-valued. In particular, the intrinsic volumes have local generalizations in the form of different versions of measures, known as {\em area measures} (defined on the unit sphere), {\em curvature measures} (defined on $\R^n$), and {\em support measures} (defined on sets of support elements). For each of these, characterization theorems of Hadwiger type, with suitably modified assumptions, have been proved, in \cite{Sch75}, \cite{Sch78}, \cite{Gla97} (see also the formulations in \cite[Sec. 4.2, Notes 11, 12]{Sch14}). Among the integral-geometric applications are a short proof of Federer's \cite{Fed59} kinematic formulas for curvature measures of convex bodies in \cite{Sch78} and a new type of kinematic formulas for support measures in \cite{Gla97} (a technical restriction was later removed in \cite{Sch99}).

The theory of valuations on convex bodies comprised already an impressive body of results, as documented by the survey articles \cite{McMS83}, \cite{McM93}, when Klain \cite{Kla95} published a new and shorter proof of Hadwiger's characterization theorem and noticed some new consequences. Klain's proof is reproduced in the book of Klain and Rota \cite{KR97}, which gives an excellent introduction to valuations and their integral-geometric applications, with side-views to some discrete aspects. (The proof is also reproduced in \cite[Thm. 6.4.14]{Sch14}). Klain's proof gave new impetus to the theory of valuations and, in particular, Hadwiger's theorem became anew the incentive and template for a large number of characterization and classification results for valuations. This second phase of valuation theory, beginning in the late 1990s, has two main branches. One of these is a deep algebraic structure theory for valuations, developed mainly by Semyon Alesker. Hints to the literature are found in the very brief sketch in \cite[Sec. 6.5]{Sch14}.
To illuminate the considerable impact that this new theory had on integral geometry, we mention the articles \cite{BF10}, \cite{BFS14} and the surveys given by Bernig \cite{Ber10} and Fu \cite{Fu11}, both under the title of `Algebraic integral geometry'. The role of measure-valued valuations in this new theory is revealed in \cite{BFS14} and \cite{Wan14}. The other branch of valuation theory, initiated by Monika Ludwig, has produced a series of important characterization theorems where the assumed invariance, covariance or contravariance is with respect to the groups ${\rm GL}(n)$ or ${\rm SL}(n)$, with or without translation invariance. We refer the reader to \cite[Sec. 10.16]{Sch14} for a brief survey with hints to the original literature.

The simple geometric nature of the properties appearing in Hadwiger's characterization theorem and the close relation between the notions of `additivity' (as defined above) and that of an `extensive property' (as coined by Richard Tolman and used in the physics of interacting particle systems) may be reasons why the intrinsic volumes have found surprising applications in statistical physics (under the name of `Minkowski functionals'). The survey by Mecke \cite{Mec00} explains how Minkowski functionals are used to describe the morphology of random spatial configurations and how they are applied to the investigation of  physical properties of materials such as complex fluids and porous media. A more recent trend employs, even more surprisingly, the natural tensor-valued generalizations of the intrinsic volumes, called `Minkowski tensors', in physics. For small dimensions and low ranks, Minkowski tensors have been applied, and are finding increasing interest, in the investigation of real materials, in particular of the morphology and anisotropy analysis of cellular, granular or porous structures. We refer (in chronological order) to \cite{BDMW02, SchT10, SchT11, SchT12, HHKM13}, for example. 

With the Minkowski tensors, we come to the central goal of the present paper. For these tensor functions, a natural extension of Hadwiger's characterization theorem has been established previously, and we now aim at a similar classification of their local versions. In early analogues of Hadwiger's characterization theorem, real-valued valuations were replaced by vector-valued valuations, resulting in characterizations of moment vectors and curvature centroids (\cite{HS71}, \cite{Sch72}). The next step of extension took some time. When McMullen \cite{McM97} initiated a thorough study of tensor-valued versions of the intrinsic volumes, he also took a possible characterization into consideration. As it turned out, Alesker was in possession of the right results from his work \cite{Ale99a} on rotation invariant valuations to prove in \cite{Ale99b} a characterization theorem for the (suitably modified) Minkowski tensors. The step from vector-valued to tensor-valued valuations with covariance properties with respect to the isometry group required new methods for their characterization. Alesker made use of representation theory for the rotation group. An approach to Alesker's characterization theorem using essentially only Hadwiger's techniques has so far not been successful.

Alesker's characterization of Minkowski tensors and the previously mentioned characterizations of the local versions of the intrinsic volumes call immediately for a classification of local versions of the Minkowski tensors, in the form of tensor-valued measures, by their most essential geometric properties. Such a local characterization theorem, which turned out to be a non-trivial task, is the subject of the present paper. One motivation for this is the expectation that the local versions of the Minkowski tensors should be more flexible for integral-geometric applications, since in contrast to the global Minkowski tensors they do not satisfy non-trivial linear relations (see the next section). A first step of a characterization theorem, namely for local tensor valuations on convex polytopes, was accomplished in \cite{Sch12}. The extension to general convex bodies is, however, far from straightforward, since unexpected local tensor valuations have come up in the polytopal case. Among these, we shall single out those admitting a continuous extension to general convex bodies. For the precise description of the situation, some more elaborate explanations are needed, which we postpone to the next section. Here we only point out that Alesker's characterization theorem will be recalled as Theorem \ref{Theorem2.1}, the local characterization theorem for polytopes will be repeated and refined in Theorem \ref{Theorem2.2}, and the main result of this paper is Theorem \ref{Theorem2.3}.

As indicated, if one wants to extend the characterization theorem for local tensor valuations from polytopes to general convex bodies, a crucial issue is whether the tensor measure valued mappings that appear in the polytopal case admit weakly continuous extensions to general convex bodies. Section \ref{sec4} provides a positive answer in certain cases, and Section \ref{sec5} gives a negative answer in the remaining cases. A suitable refinement of this negative result finally leads to a proof of the main theorem. The employed methods in both parts are very different. In Section \ref{sec4}, a tensor-valued, rotation covariant differential form is defined and evaluated at the normal cycle of a convex set. Then basic geometric measure theory is used to show that this defines a (weakly) continuous extension of some of the tensor measure valued mappings known from the polytopal case. The main tool of Section \ref{sec5} is the approximation of a highly symmetric convex body by polytopes with controllable symmetries. The approximating polytopes are constructed by lifting a polytopal complex, which is defined by a lattice in $\R^{n-1}$, to a paraboloid of revolution. For reasons indicated in Section 5, a distinction has to be made between dimensions at least four and dimension three, where the construction is more delicate.

\section{Explanations and results}\label{sec2}

We work in $n$-dimensional Euclidean space $\R^n$ ($n\ge 2$) with a fixed scalar product $\langle\cdot\,,\cdot\rangle$  and induced norm $\|\cdot\|$. The scalar product is also used to identify $\R^n$ with its dual space. We write $B^n$ for the unit ball, $\Sn$ for the unit sphere of $\R^n$, and $\Sigma^n$ for the product space $\R^n\times\Sn$.  Sometimes we identify $\R^n\times\R^n$ with $\R^{2n}$. The $k$-dimensional Hausdorff measure in a Euclidean space is denoted by $\Ha^k$. The constant $\omega_n=2\pi^{n/2}/\Gamma(n/2)$ is the $(n-1)$-dimensional Hausdorff measure of $\Sn$ and $\kappa_n=n\omega_n$ is the volume of $B^n$. 

For $p\in\N_0$, let $\T^p$ be the vector space of symmetric tensors of rank $p$ on $\R^n$. The symmetric tensor product (see, e.g., Satake \cite[Chap.~5, Sec.~4.2]{Sat75}) is denoted by $\odot$, but we shall throughout use the abbreviations
$$ a\odot b=:ab,\qquad \underbrace{a\odot\cdots\odot a}_r=:a^r$$
for symmetric tensors $a,b$ and for $r\in \N$. Moreover, $a^0:= 1$ for $a\not=0$. 
Since we have identified $\R^n$ with its dual space via the scalar product, each symmetric $p$-tensor is a symmetric $p$-linear functional on $\R^n$, and for $x\in\R^n$ and $r\in\N$ we have $x^r(y_1,\dots,y_r)=\langle x,y_1\rangle\cdots \langle x,y_r\rangle$ for $y_1,\dots,y_r\in\R^n$. 

For a topological space $X$, we write $\B(X)$ for the $\sigma$-algebra of Borel sets of $X$. If $X$ is a metric space, we denote by $\B_b(X)$ the ring of bounded Borel sets in $\B(X)$.

By $\Kn$ we denote the space of convex bodies (nonempty compact convex subsets) of $\R^n$, equipped with the Hausdorff metric. For basic facts about convex bodies used in the following, we refer to \cite{Sch14}. The subset of convex polytopes is denoted by $\Pn$. Let $K\in\Kn$. By a {\em support element} of $K$ we understand a pair $(x,u)$ where $x$ is a boundary point of $K$ and $u$ is an outer unit normal vector of $K$ at $x$. The set of all support elements of $K$ is denoted by ${\rm Nor}\,K$ and is called the {\em normal bundle} of $K$. It is a closed subset of the space $\Sigma^n$. For $x\in\R^n\setminus K$, the point $p(K,x)$ is the unique point in $K$ nearest to $x$, and the unit vector $u(K,x):= (x-p(K,x))/\|x-p(K,x)\|$ points from $p(K,x)$ to $x$. Clearly, $(p(K,x),u(K,x))\in{\rm Nor}\,K$. For $\rho>0$ and $\eta\in\B(\Sigma^n)$, the volumes of the local parallel sets
$$
M_\rho(K,\eta):=\{x\in (K+\rho B^n)\setminus K:(p(K,x),u(K,x))\in\eta\}
$$
permit a polynomial expansion of Steiner-type,
$$
\Ha^n(M_\rho(K,\eta))=\sum_{k=0}^{n-1}\rho^{n-k}\kappa_{n-k}\Lambda_{k}(K,\eta).
$$
This defines the {\em support measures}  $\Lambda_0(K,\cdot),\dots,\Lambda_{n-1}(K,\cdot)$ of $K$ (also known as generalized curvature measures). They are re-normalized versions of the measures $\Theta_k(K,\cdot)$ introduced in \cite[Sec. 4.2]{Sch14}, namely
\begin{equation}\label{2.2a}
n\kappa_{n-k}\Lambda_k(K,\cdot)=\binom{n}{k}\Theta_k(K,\cdot).
\end{equation}
The measures $\Lambda_k(K,\cdot)$ are concentrated on ${\rm Nor }\,K$. We need them here for a description of the Minkowski tensors.

For $K\in\Kn$, the {\em Minkowski tensors} are defined by
\begin{equation}\label{2.2}
\Psi_r(K) = \Phi^{r,0}_n(K):= \frac{1}{r!}\int_K x^r\,\Ha^n(\D x)
\end{equation}
and
\begin{equation}\label{2.3}
\Phi_k^{r,s}(K):= c_{n,k}^{r,s} \int_{\Sigma^n} x^ru^s\,\Lambda_k(K,\D(x,u))
\end{equation}
with 
$$  c_{n,k}^{r,s} := \frac{1}{r!s!}\frac{\omega_{n-k}}{\omega_{n-k+s}}$$
for $r,s\in\N_0$ and $k\in\{0,\dots,n-1\}$; for convenience, we put $\Phi_k^{r,s}=0$ for all other $r,s,k$. We remark that the moment tensor $\Psi_r$ is a very natural construction and that the tensors (\ref{2.3}) necessarily come up if $\Psi_r$ is applied to a parallel body. In fact, for $\rho>0$, the Steiner-type formula
$$ \Psi_r(K+\rho B^n) = \sum_{k=0}^{n+r} \rho^{n+r-k}\kappa_{n+r-k}\sum_{s\in{\mathbb N}_0}\Phi_{k-r+s}^{r-s,s}(K)$$
holds, which was proved in \cite{Sch00} with other notations. See also \cite[Subsection 5.4.2]{Sch14}.

Each Minkowski tensor $\Phi_k^{r,s}$ defines a mapping $\Gamma:\Kn\to\T^p$, for $p=r+s$, which is a valuation and is continuous (with respect to the topology on $\Kn$ induced by the Hausdorff metric and the standard topology on $\T^p$).  Moreover, $\Gamma$ is isometry covariant, that is, it is rotation covariant and has polynomial translation behavior. Here, rotation covariance is defined by $\Gamma(\vartheta K)=\vartheta\Gamma(K)$ for $\vartheta\in{\rm O}(n)$, where the rotation group ${\rm O}(n)$ operates in the standard way on $\T^p$, namely by 
$$ (\vartheta T)(x_1,\dots,x_p)= T(\vartheta^{-1}x_1,\dots,\vartheta^{-1}x_p)$$
for $T\in\T^p$ and $x_1,\dots,x_p\in\R^n$. Polynomial translation behavior of $\Gamma$ means that $\Gamma(K+t)$ is a tensor polynomial in $t\in\R^n$, that is, there are tensors $\Gamma_{p-j}(K)\in\T^{p-j}$, independent of $t$,  such that
$$ \Gamma(K+t)= \sum_{j=0}^p \Gamma_{p-j}(K)t^j$$
for all $t\in\R^n$ and all $K\in\Kn$. (We define $0^0:=1$ here.)

The listed properties are sufficient to essentially characterize the Minkowski tensors. Here `essentially' refers to the facts that linear combinations with constant coefficients preserve the properties and that the metric tensor $Q$, defined by
$$ Q(x,y):= \langle x,y\rangle\qquad\mbox{for }x,y\in \R^n,$$
is also isometry covariant. Therefore, multiplication by a power of the metric tensor also preserves the listed properties. The following characterization theorem was proved in \cite{Ale99b}.

\begin{theorem}[Alesker]\label{Theorem2.1}
Let $p\in{\mathbb N}_0$. The real vector space of continuous, isometry covariant valuations on ${\mathcal K}^n$ with values in ${\mathbb T}^p$ is spanned by the tensor valuations $Q^m\Phi_k^{r,s}$, where $m,r,s \in{\mathbb N}_0$ satisfy
$2m+r+s=p$ and where $k\in\{0,\dots,n\}$, but $s=0$ if $k=n$.
\end{theorem}

The dimensions of the vector spaces of continuous, isometry covariant tensor valuations of a fixed rank and a given degree of homogeneity were explicitly determined in \cite{HSS08a}. This required some effort, since there exist non-trivial linear relations between the tensor functions $Q^m\Phi_k^{r,s}$, which had been discovered by McMullen \cite{McM97}. The existence of these linear relations was an obstruction to applying the characterization theorem in Hadwiger's fashion, to derive integral-geometric formulas. Instead, such formulas for Minkowski tensors were proved by direct, cumbersome computations, in \cite{HSS08b}. Recently, the modern structure theory of valuations provided a new approach to part of these and additional new integral-geometric formulas for tensor valuations, see \cite{BH14}.

\vspace{2mm}

The natural local versions of (\ref{2.3}), which we call {\em local Minkowski tensors}, are defined by
\begin{equation}\label{2.4}
\phi_k^{r,s}(K,\eta):= c_{n,k}^{r,s} \int_\eta x^ru^s\,\Lambda_k(K,\D(x,u))
\end{equation}
for $\eta\in\B(\Sigma^n)$ and $r,s\in\N_0$, $k\in\{0,\dots,n-1\}$. Each local Minkowski tensor $\phi_k^{r,s}$ defines a mapping from $\Kn\times\B(\Sigma^n)$ into $\T^p$, for $p=r+s$.

Generally, for a mapping $\Gamma:\Kn\times \B(\Sigma^n) \to \T^p$ we consider the following properties. Here we write $\eta+t:=\{(x+t,u):(x,u)\in\eta\}$ and $\vartheta \eta:= \{(\vartheta x,\vartheta u):(x,u)\in\eta\}$ for $\eta\in\B(\Sigma^n)$, $t\in\R^n$ and $\vartheta\in{\rm O}(n)$.

\noindent $\bullet$\; $\Gamma$ is {\em translation covariant of degree} $q$, where $0\le q\le p$, if
\begin{equation}\label{1} 
\Gamma(K+t,\eta+t)= \sum_{j=0}^q \Gamma_{p-j}(K,\eta)\frac{t^j}{j!}
\end{equation}
with tensors $\Gamma_{p-j}(K,\eta)\in\T^{p-j}$, for all $K\in\Kn$, $\eta\in\B(\Sigma^n)$ and $t\in\R^n$ (the denominator $j!$ appears for convenience); here $\Gamma_p=\Gamma$. In particular, $\Gamma$ is called {\em translation invariant} if it is translation covariant of degree zero. 

\noindent $\bullet$\; $\Gamma$ is {\em rotation covariant} if $\Gamma(\vartheta K,\vartheta\eta)= \vartheta \Gamma(K,\eta)$ for all $K\in\Kn$, $\eta\in\B(\Sigma^n)$ and $\vartheta\in{\rm O}(n)$.

\noindent $\bullet$\; $\Gamma$ is {\em isometry covariant} (of degree $q$) if it is translation covariant of some degree $q\le p$ (and hence of degree $p$) and rotation covariant.

\noindent $\bullet$\; $\Gamma$ is {\em locally defined} if for $\eta\in {\mathcal B}(\Sigma^n)$ and $K,K'\in{\mathcal K}^n$ with $\eta\cap {\rm Nor}\,K = \eta\cap {\rm Nor}\,K'$ the equality $\Gamma(K,\eta)=\Gamma(K',\eta)$ holds. 

\noindent $\bullet$\; If $\Gamma(K,\cdot)$ is a $\T^p$-valued measure for each $K\in\Kn$, then $\Gamma$ is {\em weakly continuous} if for each sequence $(K_i)_{i\in\N}$ of convex bodies in $\Kn$ converging to a convex body $K$ the relation
$$ \lim_{i\to\infty} \int_{\Sigma^n} f\,\D\Gamma(K_i,\cdot) = \int_{\Sigma^n} f\,\D\Gamma(K,\cdot)$$
holds for all continuous functions $f:\Sigma^n\to\R$ (the integral is defined coordinate-wise).

In the previous definitions, the set ${\mathcal K}^n$ may be replaced by ${\mathcal P}^n$.

The particular mapping $\Gamma= \phi_k^{r,s}$ has the following properties, as a consequence of the known properties of the support measures (which are found in \cite[Sec.~4.2]{Sch14}). For each $K\in\Kn$, $\Gamma(K,\cdot)$ is a $\T^p$-valued measure, and $\Gamma$ is weakly continuous. For each $\eta\in\B(\Sigma^n)$, $\Gamma(\cdot,\eta)$ is measurable and is a valuation. The mapping $\Gamma$ is isometry covariant. In fact, the translation covariance follows from
\begin{eqnarray}
\phi_k^{r,s}(K+t,\eta+t)&=& c_{n,k}^{r,s} \int_{\eta+t} x^ru^s\,\Lambda_k(K+t,\D(x,u))\nonumber\\
& =& c_{n,k}^{r,s} \int_\eta (x+t)^ru^s\,\Lambda_k(K,\D(x,u))\nonumber\\
&=& \sum_{i=0}^r\phi_k^{r-i,s}(K,\eta)\frac{t^i}{i!}. \label{2.7}
\end{eqnarray}
Finally, the mapping $\Gamma$ is locally defined.

As mentioned, these properties follow from the corresponding properties of the support measures. If one replaces $\Kn$ by the space $\Pn$ of polytopes, then it was noted by Glasauer \cite[Lem.~1.3]{Gla97} that the latter properties,  without the valuation property and the weak continuity, are already sufficient to characterize the linear combinations of the support measures. This was one of the motivations for the following considerations about local Minkowski tensors of polytopes.

For a polytope $P\in\Pn$, we denote by $\F_k(P)$ the set of $k$-dimensional faces of $P$, for $k\in\{0,\dots,n\}$. For $F\in\F_k(P)$, the set $\nu(P,F)=N(P,F)\cap\Sn$ is the set of outer unit normal vectors of $P$ at its face $F$ (see \cite[Sec.~2.4]{Sch14} for the normal cone $N(P,F)$). The local Minkowski tensors of a polytope $P$ have the explicit representation
\begin{equation}\label{2.6a} 
\phi_k^{r,s}(P,\eta)= C_{n,k}^{r,s} \sum_{F\in\F_k(P)} \int_F \int_{\nu(P,F)} {\bf 1}_\eta(x,u) x^r u^s\, 
\Ha^{n-k-1}(\D u)\,\Ha^k(\D x)
\end{equation}
with 
\begin{equation}\label{2.7a}
C_{n,k}^{r,s}:= (r!s!\omega_{n-k+s})^{-1}
\end{equation}
and with ${\bf 1}_\eta$ denoting the indicator function of the set $\eta$. This follows from a corresponding representation of the support measures, see \cite[(4.3)]{Sch14} and (\ref{2.2a}).

The local Minkowski tensors of polytopes can be generalized while preserving their properties, except weak continuity. Let $L\subset \R^n$ be a linear subspace, let $\pi_L:\R^n\to L$ be the orthogonal projection and define $Q_L\in\T^2$ by
$$ Q_L(a,b):=\langle \pi_L a,\pi_L b\rangle,\qquad \mbox{for }a,b\in\R^n.$$
Then $Q_{\vartheta L}=\vartheta Q_L$ for $\vartheta\in{\rm O}(n)$.

For $P\in\Pn$ and $F\in\F_k(P)$, let $L(F)$ be the linear subspace parallel to ${\rm aff}\, F$ (the {\em direction space} of $F$). Then we define the {\em generalized local Minkowski tensor} 
\begin{equation}\label{2.6}  
\phi_k^{r,s,j}(P,\eta) := C_{n,k}^{r,s} \sum_{F\in\F_k(P)} Q_{L(F)}^{\hspace*{1pt}j}\int_F \int_{\nu(P,F)} {\bf 1}_\eta(x,u) x^r u^s\, \Ha^{n-k-1}(\D u)\,\Ha^k(\D x),
\end{equation}
for $r,s,j,k\in\N_0$ with $1\le k\le n-1$. This definition is supplemented by $\phi_0^{r,s,0}:=\phi_0^{r,s}$, but $\phi_0^{r,s,j}$ will 
remain undefined for $j\ge 1$. 
Each mapping $\Gamma=  \phi_k^{r,s,j}$ has the following properties. It is isometry covariant and locally defined. For each $P\in\Pn$, $\Gamma(P,\cdot)$ is a $\T^p$-valued measure. For each $\eta\in\B(\Sigma^n)$, $\Gamma(\cdot,\eta)$ is a valuation. This was stated without proof in \cite{Sch12}. We shall provide a proof in the next section (Theorem \ref{T3.3}).

Not all of the listed properties are needed for a characterization.

\begin{theorem}\label{Theorem2.2}
For $p\in\N_0$, let $T_p(\Pn)$ denote the real vector space of all mappings $\Gamma:\Pn\times\B(\Sigma^n)\to\T^p$ with the following properties.\\[1mm]
$\rm (a)$ $\Gamma(P,\cdot)$ is a $\T^p$-valued measure, for each $P\in\Pn$;\\[1mm]
$\rm (b)$ $\Gamma$ is isometry covariant;\\[1mm]
$\rm (c)$ $\Gamma$ is locally defined.\\[1mm]
Then a basis of $T_p(\Pn)$ is given by the mappings $Q^m\phi^{r,s,j}_k$, where $m,r,s,j\in\N_0$ satisfy $2m+2j+r+s=p$ and where $k\in\{0,\dots,n-1\}$, but $j=0$ if $k\in\{0,n-1\}$.
\end{theorem}

This is a stronger version of a theorem proved in \cite{Sch12}, including the linear independence result of Theorem \ref{Theorem3.1}. We shall explain the further modifications, in comparison to \cite{Sch12}, in the next section.

A similar theorem for $\Kn$ instead of $\Pn$ can hardly be expected without a continuity assumption. This raises the question whether the modified local Minkowski tensors $\phi^{r,s,j}_k$ with $k\ge 1$ and $j\ge 1$ have weakly continuous extensions from $\Pn$ to $\Kn$. The answer is easily seen to be positive for $k=n-1$, see Lemma \ref{Lemma3.4}. In Section \ref{sec4} we shall show that the answer is positive for $j=1$, and we shall suggest a representation of $\phi_k^{r,s,1}(K,\cdot)$ for general convex bodies $K$. For $j\ge 2$ and $1\le k\le n-2$ the answer is negative, in a stronger sense, which will allow us to prove in Section \ref{sec5} the following main result.

\begin{theorem}\label{Theorem2.3}
For $p\in\N_0$, let $T_p(\Kn)$ denote the real vector space of all mappings $\Gamma:\Kn\times\B(\Sigma^n)\to\T^p$ with the following properties.\\
$\rm (a)$ $\Gamma(K,\cdot)$ is a $\T^p$-valued measure, for each $K\in\Kn$;\\ 
$\rm (b)$ $\Gamma$ is isometry covariant;\\
$\rm (c)$ $\Gamma$ is locally defined;\\
$\rm (d)$ $\Gamma$ is weakly continuous.\\
Then a basis of $T_p(\Kn)$ is given by the mappings $Q^m\phi^{r,s,j}_k$, where $m,r,s\in\N_0$ and $j\in\{0,1\}$ satisfy $2m+2j+r+s=p$ and where $k\in\{0,\dots,n-1\}$, but $j=0$ if $k\in\{0,n-1\}$.
\end{theorem}

\section{Reductions and the polytopal case}\label{sec3}

We assume that $\Gamma:\Kn\times\B(\Sigma^n)\to\T^p$ is a mapping which has the following properties.\\
$\bullet$ For each $K\in\Kn$, $\Gamma(K,\cdot)$ is a $\T^p$-valued measure;\\
$\bullet$ $\Gamma$ is isometry covariant;\\
$\bullet$ $\Gamma$ is locally defined.\\
Here $\Kn$ may be replaced by $\Pn$.

It is our goal to classify the mappings $\Gamma$ with these and possibly a continuity property. As a preparation, 
in the present section we establish several auxiliary results and provide reduction steps in order to deduce the general results from some simpler classification problems.

If $\Gamma$ is translation covariant of degree $q$, then there are mappings $\Gamma_{p-j}:\Kn\times\B(\Sigma^n)\to\T^{p-j}$, $j=0,\dots,q$, (possibly zero for some $j$ and with $\Gamma_p=\Gamma$) such that
$$ \Gamma(K+t,\eta+t)= \sum_{j=0}^q \Gamma_{p-j}(K,\eta)\frac{t^j}{j!}$$
for all $K\in \Kn$, $\eta\in \B(\Sigma^n)$ and $t\in\R^n$. The following lemma extends an observation of McMullen \cite[Thm.~2.3]{McM97}. 
 
\begin{lemma}\label{Lemma3.1}
If $\Gamma$ is translation covariant of degree $q\le p$, then the mappings $\Gamma_{p-j}$ satisfy
$$ \Gamma_{p-j}(K+t,\eta+t) =\sum_{r=0}^{q-j} \Gamma_{p-j-r}(K,\eta)\frac{t^r}{r!}$$
for $j=0,\dots,q$ and all $K\in \Kn$, $\eta\in \B(\Sigma^n)$ and $t\in\R^n$, in particular (case $j=q$),%}
$$ \Gamma_{p-q}(K+t,\eta+t) = \Gamma_{p-q}(K,\eta).$$
\end{lemma}

\begin{proof}
For $s,t\in\R^n$ we have
\begin{eqnarray*}
& & \sum_{j=0}^q \Gamma_{p-j}(K+t,\eta+t)\frac{s^j}{j!} = \Gamma(K+t+s,\eta+t+s) = \sum_{i=0}^q \Gamma_{p-i}(K,\eta)\frac{(t+s)^i}{i!}\\
& & = \sum_{i=0}^q \Gamma_{p-i}(K,\eta) \sum_{j=0}^i\frac{t^{i-j}}{(i-j)!} \frac{s^j}{j!} = \sum_{j=0}^q\left(\sum_{i=j}^q\Gamma_{p-i}(K,\eta)\frac{t^{i-j}}{(i-j)!}\right)\frac{s^j}{j!}.
\end{eqnarray*}
It follows that
$$ \Gamma_{p-j}(K+t,\eta+t)=\sum_{i=j}^q \Gamma_{p-i}(K,\eta)\frac{t^{i-j}}{(i-j)!} = \sum_{r=0}^{q-j} \Gamma_{p-j-r}(K,\eta) \frac{t^r}{r!}.$$
Here we have used the subsequent lemma, together with the fact that the symmetric tensor algebra has no zero divisors.
\end{proof}

In order to derive properties of $\Gamma_{p-j}$ from those of $\Gamma$, the following lemma is useful. It is simpler than \cite[Lem.~1]{Sch12}, which was also used for that purpose.

\begin{lemma}\label{Lemma3.2}
Let $\Gamma$ be translation covariant of degree $q\le p$. Then there are constants $a_{jm}$ $($$j=0,\dots,q$, $m=1,\dots,q+1$$)$, depending only on $q,j,m$, such that
\begin{equation}\label{3} 
\Gamma_{p-j}(K,\eta)\frac{t^j}{j!} =\sum_{m=1}^{q+1} a_{jm}\Gamma(K+mt,\eta+mt)
\end{equation}
for all $K\in \Kn$, $\eta\in \B(\Sigma^n)$ and $t\in\R^n$.
\end{lemma}

\begin{proof} For fixed $K,\eta,t$, let $f(\lambda)$, for $\lambda\in\R$, be a coordinate of $\Gamma(K+\lambda t,\eta+\lambda t)$ with respect to some given basis, and let $f_j$ be the corresponding coordinate of $\Gamma_{p-j}(K,\eta) t^j/j!$. (The following argument is similar to one used in \cite[p.~213]{Sch14}.) In the equation $f(\lambda)=\sum_{j=0}^q \lambda^j f_j$, which holds by (\ref{1}), we insert for $\lambda$ the values $1,\dots,q+1$. The resulting system of linear equations for $f_0,\dots,f_q$ has a (Vandermonde) determinant different from zero, hence there is a solution of the form $f_j=\sum_{m=1}^{q+1} a_{jm}f(m)$, $j=0,\dots,q$, with certain constants $a_{jm}$, depending only on $q,j,m$. Since this holds for all coordinates, equation (\ref{3}) results. 
\end{proof}

A typical application is as follows. From the isometry covariance of $\Gamma$ we get, for $\vartheta\in{\rm O}(n)$,
\begin{eqnarray*}
\Gamma_{p-j}(\vartheta K,\vartheta\eta)\frac{(\vartheta t)^j}{j!} &= &\sum_{m=1}^{q+1} a_{jm} \Gamma(\vartheta(K+mt),\vartheta(\eta+mt))\\
& = &\vartheta\left( \sum_{m=1}^{q+1} a_{jm} \Gamma(K+mt,\eta+mt)\right) = \vartheta\left( \Gamma_{p-j}(K,\eta) \frac{t^j}{j!}\right)\\
&=& \vartheta \Gamma_{p-j}(K,\eta)\frac{(\vartheta t)^j}{j!}.
\end{eqnarray*}
Since the symmetric tensor algebra has no zero divisors, it follows that
$$ \Gamma_{p-j}(\vartheta K,\vartheta\eta) = \vartheta\Gamma_{p-j}(K,\eta).$$
Together with Lemma \ref{Lemma3.1}, this shows that $\Gamma_{p-j}$ is isometry covariant.

In a similar way, one shows that $\Gamma_{p-j}(K,\cdot)$ is a $\T^{p-j}$-valued measure, for each $K\in\Kn$. Further, if $\Gamma$ is weakly continuous or is a valuation in its first argument, then each $\Gamma_{p-j}$ has the corresponding property.

The following lemma extends an argument from \cite[p.~124]{Sch78}. A similar argument (in a simpler situation) appears in \cite{Gla97} (proof of Lemma 1.3).

\begin{lemma}\label{Lemma3.3}
For each $K\in \Kn$, the measure $\Gamma(K,\cdot)$ is concentrated on ${\rm Nor}\,K$.
\end{lemma}

\begin{proof} Recall that $\B_b(\Sigma^n)$ is the ring of bounded Borel sets in $\Sigma^n$. Let $\eta\in \B_b(\Sigma^n)$. Then we can choose points $x_1,x_2\in\R^n$ with $\eta\cap{\rm Nor}\,\{x_i\}=\emptyset$ for $i=1,2$. Since $\Gamma$ is locally defined, we have $\Gamma(\{x_1\},\eta)= \Gamma(\{x_2\},\eta)$. Therefore, we can define
$$ F(\eta):= \Gamma(\{x\},\eta) \quad \mbox{for }\eta\in{\mathcal B}_b(\Sigma^n) \mbox{ and arbitrary $x$ with  } \eta\cap{\rm Nor}\,\{x\}=\emptyset.$$
Let $(\eta_i)_{i\in{\mathbb N}}$ be a disjoint sequence in ${\mathcal B}_b(\Sigma^n)$ such that $\bigcup_{i\in{\mathbb N}} \eta_i\in{\mathcal B}_b(\Sigma^n)$. Then we can choose $x$ with $\eta_i\cap{\rm Nor}\,\{x\}=\emptyset$ for all $i\in{\mathbb N}$ and deduce that 
$$ \sum_{i\in\N} F(\eta_i)= \sum_{i\in\N} \Gamma\left(\{x\},\eta_i\right)=\Gamma\left(\{x\},\bigcup_{i\in\N} \eta_i\right) = F\left(\bigcup_{i\in\N} \eta_i\right).$$ 
Thus, $F$ is a ${\mathbb T}^p$-valued measure on the ring ${\mathcal B}_b(\Sigma^n)$.

Let $\eta\in{\mathcal B}_b(\Sigma^n)$ and $t\in{\mathbb R}^n$. Choosing $x$ with $\eta\cap{\rm Nor}\,\{x\}=\emptyset$, we have $(\eta+t)\cap{\rm Nor}\,\{x+t\}=\emptyset$ and hence
$$ F(\eta+t)= \Gamma(\{x\}+t,\eta+t) = \sum_{j=0}^p \Gamma_{p-j}(\{x\},\eta)\frac{t^j}{j!}.$$
This is independent of $x$ (as long as $\eta\cap{\rm Nor}\,\{x\}=\emptyset$), hence we can define
$$ F_0(\eta):= \Gamma_0(\{x\},\eta).$$
From Lemma \ref{Lemma3.1} it follows that $F_0(\eta+t)=F_0(\eta)$. Since $\Gamma_{p-j}(\{x\},\cdot)$ is a ${\mathbb T}^{p-j}$-valued measure, $F_0$ is a real-valued signed measure on ${\mathcal B}_b(\Sigma^n)$. Let $\omega\in{\mathcal B}({\mathbb S}^{n-1})$ be fixed and define
$$ \mu(\beta):= F_0(\beta\times \omega)\qquad\mbox{for } \beta\in{\mathcal B}_b({\mathbb R}^n).$$
Then $\mu$ is a translation invariant finite signed measure on ${\mathcal B}_b({\mathbb R}^n)$ and hence a multiple of Lebesgue measure. Thus, there is a constant $c$ such that $\Gamma_0(\{x\}, \beta\times\omega)=c\,\Ha^n(\beta)$ for all bounded Borel sets $\beta$ with $x\notin\beta$, but since both sides are measures in $\beta$, the equality holds for arbitrary Borel sets $\beta$ with $x\notin\beta$. If $c\not=0$, then $\Gamma_0(\{x\}, \beta\times\omega)=\infty$ if, for instance, $\beta$ is a halfspace and $x$ a point not contained in it. This is a contradiction. From $F_0(\beta\times\omega)=0$ for all $\beta\in{\mathcal B}({\mathbb R}^n)$ and  $\omega\in{\mathcal B}({\mathbb S}^{n-1})$ it follows that $F_0(\eta)=0$ for all $\eta\in{\mathcal B}(\Sigma^n)$. Since $F_0=0$, we now have
$$ F(\eta+t)= \sum_{j=0}^{p-1} \Gamma_{p-j}(\{x\},\eta)\frac{t^j}{j!}$$
for $\eta\in{\mathcal B}_b(\Sigma^n)$, $x\in{\mathbb R}^n$ with $\eta\cap{\rm Nor}\{x\}=\emptyset$ and $t\in{\mathbb R}^n$.
Repeating the argument above and arguing as in the proof of Lemma \ref{Lemma3.1}, we obtain that $F_1(\eta):= \Gamma_1(\{x\},\eta)$ defines a translation invariant ${\mathbb T}^1$-valued measure $F_1$ on ${\mathcal B}_b(\Sigma^n)$. As above, we deduce that each coordinate of $F_1$ must be zero, and from $F_1=0$ we conclude that
$$ F(\eta+t)= \sum_{j=0}^{p-2} \Gamma_{p-j}(\{x\},\eta)\frac{t^j}{j!}.$$
The argument can now be repeated and after finitely many steps we arrive at $F=0$.

Now let $K\in{\mathcal K}^n$ and let $\eta\in{\mathcal B}_b(\Sigma^n)$ be a set with $\eta\cap {\rm Nor}\,K=\emptyset$. We can choose a point $x$ with $\eta\cap{\rm Nor}\,\{x\}=\emptyset$. Since $\Gamma$ is locally defined, it follows that
$$ \Gamma(K,\eta)=\Gamma(\{x\},\eta)=0.$$
From this, we can deduce that $\Gamma(K,\eta)=0$ for arbitrary sets $\eta\in{\mathcal B}(\Sigma^n)$ with $\eta\cap{\rm Nor}\,K=\emptyset$. This means that $\Gamma(K,\cdot)$ is concentrated on ${\rm Nor}\,K$.
\end{proof}

Now we turn to polytopes and remark first that all of the previous statements of this section remain true if $\Kn$ is replaced by $\Pn$.

We study the generalized local Minkowski tensors $\phi_k^{r,s,j}$ of polytopes defined by (\ref{2.6}). For $k=n-1$, we show that they yield nothing new, as they can be expressed as linear combinations of the mappings $Q^m\phi_{n-1}^{r,l}$.

\begin{lemma}\label{Lemma3.4}
\begin{equation}\label{5}
\phi_{n-1}^{r,s,j} = \sum_{i=0}^j (-1)^i\binom{j}{i}\frac{(s+2i)!\omega_{1+s+2i}}{s!\omega_{1+s}} Q^{j-i}\phi_{n-1}^{r,s+2i}.
\end{equation}
\end{lemma}

\begin{proof} Let $P\in{\mathcal P}^n$. For $F\in{\mathcal F}_{n-1}(P)$, let $\pm u_F$ be the two unit vectors orthogonal to $L(F)$. Then $Q_{L(F)}+u_F^2=Q$, hence
$$  \phi_{n-1}^{r,s,j}(P,\eta) = C_{n,n-1}^{r,s} \sum_{F\in{\mathcal F}_{n-1}(P)} \int_F \int_{\nu(P,F)}\left(Q-u_F^2\right)^j {\bf 1}_\eta(x,u) x^r u^s\, {\mathcal H}^0(\D u)\,{\mathcal H}^{n-1}(\D x).$$
This immediately gives the assertion.
\end{proof}

Whereas the global tensor functions $Q^m\Phi_k^{r,s}$ satisfy the non-trivial linear `McMullen relations', their local versions on polytopes, $Q^m\phi_k^{r,s}$ with $0\le k\le n-1$ and $Q^m\phi_k^{r,s,j}$ with $1\le k\le n-2$ and $j\ge 1$,  are linearly independent. This is the subject of the following theorem, which is part of Theorem \ref{Theorem2.2}. It is expected that this fact will later be useful for the derivation of integral-geometric relations.

\begin{theorem}\label{Theorem3.1}
Let $p\in{\mathbb N}_0$. On $\Pn$, the local tensor valuations $Q^m\phi_k^{r,s,j}$ with 
$$ m,r,s,j\in{\mathbb N}_0,\;2m+2j+r+s=p,\;k\in\{0,\dots,n-1\}, \mbox{ but }j=0 \mbox{ if } k\in\{0,n-1\},$$
are linearly independent.
\end{theorem}

\begin{proof} Suppose that
\begin{equation}\label{7} 
\sum_{m,r,s,j,k \atop 2m+2j+r+s=p} a_{kmrsj} Q^m\phi_k^{r,s,j}=0
\end{equation}
with $a_{kmrsj}\in{\mathbb R}$ and with $a_{0mrsj}=a_{(n-1)mrsj}=0$ for $j\not=0$.

Let $F$ be a $k$-dimensional polytope, $k\in\{0,\dots,n-1\}$. Let $\text{relint } M$ denote the relative interior of a convex set $M$ in $\R^n$. For arbitrary Borel sets $\beta\subset \text{relint }F$ and $\omega\subset L(F)^\perp\cap{\mathbb S}^{n-1}$ (where $L^\perp$ denotes the orthogonal complement of $L$) we have
$$ \phi_k^{r,s,j}(F,\beta\times\omega) =  Q_{L(F)}^j C_{n,k}^{r,s} \int_\beta x^r\,{\mathcal H}^k(\D x) \int_\omega u^s \, {\mathcal H}^{n-k-1}(\D u) $$
and $\phi_{k'}^{r,s,j}(F,\beta\times\omega) =0$ for $k'\not= k$. Therefore, for each $k\in\{0,\dots,n-1\}$ we get
$$ \sum_{m,r,s,j \atop 2m+2j+r+s=p} a_{kmrsj}Q^m Q_{L(F)}^j C_{n,k}^{r,s}\int_\beta x^r\,{\mathcal H}^k(\D x) \int_\omega u^s\,{\mathcal H}^{n-k-1}(\D u)=0$$
whenever $F$ is a $k$-dimensional polytope and $\beta\subset \text{relint }  F$, $\omega\subset L(F)^\perp\cap{\mathbb S}^{n-1}$ are  Borel sets.

We can choose a translate of $F,\beta$ such that $\int_\beta x^r\,{\mathcal H}^k(\D x)\not=0$. Replacing $F,\beta$ by multiples and comparing degrees of homogeneity, we conclude that for each fixed $r\in{\mathbb N}_0$ we have
$$ \sum_{m,s,j \atop 2m+2j+s=p-r} a_{kmrsj}Q^m Q_{L(F)}^j C_{n,k}^{r,s}\int_\beta x^r\,{\mathcal H}^k(\D x) \int_\omega u^s\,{\mathcal H}^{n-k-1}(\D u)=0.$$

Since the symmetric tensor algebra has no zero divisors, it follows that
$$ \sum_{m,s,j \atop 2m+2j+s=p-r} a_{kmrsj}Q^m Q_{L(F)}^j C_{n,k}^{r,s} \int_\omega u^s\,{\mathcal H}^{n-k-1}(\D u)=0.$$
This holds for arbitrary Borel sets $\omega \subset {\mathbb S}^{n-1}\cap L(F)^\perp$, hence we obtain
$$ \sum_{m,s,j \atop 2m+2j+s=p-r} a_{kmrsj}Q^m Q_{L(F)}^j C_{n,k}^{r,s} u^s=0$$
for all $u\in {\mathbb S}^{n-1}\cap L(F)^\perp$. Here $L(F)$ can be an arbitrary $k$-dimensional linear subspace of $\R^n$. Since $n$, $k$ and $r$ are fixed, we set $a_{kmrsj} C_{n,k}^{r,s}=:b_{msj}$ and write the latter relation in the form
\begin{equation}\label{4}
\sum_{m,s,j \atop 2m+2j+s=p-r} b_{msj}Q^m Q_{L(F)}^j  u^s=0.
\end{equation}

Let $(e_1,\dots,e_n)$ be an orthonormal basis of ${\mathbb R}^n$ such that $L(F)$ is the subspace spanned by $e_1,\dots,e_k$. Applying both sides of (\ref{4}) to the $(p-r)$-tuple
$$ (\underbrace{x,\dots,x}_{2m+2j+s}),\qquad x=x_1e_1+\cdots+x_ne_n\in\R^n,$$
we obtain
$$ \sum_{m,s,j \atop 2m+2j+s=p-r} b_{msj} (x_1^2+\dots+x_n^2)^m(x_1^2+\dots+x_k^2)^j(u_{k+1} x_{k+1}+\dots+u_nx_n)^s=0.$$
This holds for all $x_1,\dots,x_n\in{\mathbb R}$ and all $u_{k+1},\ldots,u_n\in\R$ such that 
$u_{k+1}e_{k+1}+\dots +u_ne_n\in{\mathbb S}^{n-1}$. If $k=0$, then $b_{msj}=0$ for $j\neq 0$ and for $j=0$ 
the factor $(x_1^2+\dots+x_k^2)^j$ is omitted (respectively, taken as 1). 

First we assume now that $0\le k\le n-2$. For fixed $x_{k+1},\dots,x_n\not=0$, the values
$$ t = u_{k+1} x_{k+1}+\dots+u_nx_n$$
with $u_{k+1}e_{k+1}+\dots+u_ne_n\in{\mathbb S}^{n-1}$ fill a nondegenerate interval, hence it follows that
$$ \sum_{m,j \atop 2m+2j=p-r-s} b_{msj}(x_1^2+\dots+ x_n^2)^m(x_1^2+\dots+ x_k^2)^j=0$$
for all $s\in{\mathbb N}_0$. Only such $s$ occur here for which $p-r-s$ is even, say equal to $2q$. For fixed $s$, the latter relation can be written in the form 
$$ \sum_{m=0}^q b_m (x_1^2+\dots+ x_n^2)^m(x_1^2+\dots+ x_k^2)^{q-m}=0$$
with $b_m:= b_{ms(q-m)}$. It holds for all $x_1,\dots,x_n\not=0$ and hence for arbitrary $x_1,\dots, x_n$. Since the highest appearing power of $x_n$ must appear with coefficient zero, we have $b_q=0$, then $b_{q-1}=0$, and so on until $b_0=0$. Thus, all coefficents $a_{kmrsj}$ in (\ref{7}) with $k\le n-2$ are zero.

It remains to consider $k=n-1$. In this case, we have
$$ b_{msj}=a_{(n-1)mrsj}C_{n,n-1}^{r,s}=0 \quad\mbox{for }j\not=0,$$
hence (\ref{4}) reduces to 
$$ \sum_{m,s \atop 2m+s=p-r} b_{ms0}Q^m u^s=0.$$
Similarly as above, this implies that 
$$ \sum_{m=0}^{\lfloor \frac{q}{2}\rfloor} b_{m(q-2m)0}(x_1^2+\dots+ x_n^2)^m x_n^{q-2m}=0$$
for all $x_1,\ldots,x_n\in\R$ with $|x_n|= 1$, where $q=p-r$. Comparing powers of $x_1$, we conclude that all coefficients $b_{m( q-2m) 0}$ must be zero. Thus, also all coefficients $a_{(n-1)mrsj}$ in (\ref{7}) are zero. 
\end{proof}

The local characterization for polytopes, Theorem  \ref{Theorem2.2} without the assertion about linear independence, was essentially proved in \cite{Sch12}, as mentioned in the previous section. One of the further differences is that in \cite{Sch12} it was additionally assumed that $\Gamma(P,\cdot)$ is concentrated on ${\rm Nor}\,P$. By Lemma \ref{Lemma3.3}, this follows from the other assumptions. Another difference is the fact that the mappings corresponding to $\phi_k^{r,s,j}$ in 
\cite[Thm.~1]{Sch12} involve (with slightly different notations) tensors of the form $Q_L^lQ_{L^\perp}^{m-l}$. But the latter is a linear combination of tensors of the form $Q^iQ_L^j$, since $Q_L+Q_{L^\perp}=Q$. Further, we point out that in Theorem \ref{Theorem2.3} the terms $\phi_{n-1}^{r,s,j}$ with $j\ge 1$ are not really needed, due to Lemma \ref{Lemma3.4}.

Next, we wish to point out that the proof of \cite[Thm. 1]{Sch12}, and thus the proof of Theorem \ref{Theorem2.2}, can be simplified slightly. We consider first the translation invariant case, then we show how the general case can be deduced from this special one.

\begin{theorem}\label{Theorem3.2}
Let $p\in{\mathbb N}_0$. Let $\Gamma:\Pn\times\B(\Sigma^n)\to\T^p$ be a mapping with the following properties.\\
$\rm (a)$ $\Gamma(P,\cdot)$ is a $\T^p$-valued measure, for each $P\in\Pn$; \\
$\rm (b)$ $\Gamma$ is translation invariant and rotation covariant;\\
$\rm (c)$ $\Gamma$ is locally defined.\\
Then $\Gamma$ is a linear combination, with constant coefficients, of the mappings $Q^m\phi^{0,s,j}_k$, where $m,s,j\in\N_0$ satisfy $2m+2j+s=p$ and where $k\in\{0,\dots,n-1\}$, but $j=0$ if $k\in\{0,n-1\}$.
\end{theorem}

\begin{proof} We sketch the proof, to show where the proof given in \cite{Sch12} can be simplified. We have to prove (for each polytope $P$) the equality of two tensor-valued measures on $\B(\Sigma^n)$. It is sufficient to prove equality on sets of the form $\beta\times \omega$ with $\beta\in\B(\R^n)$ and $\omega\in\B(\Sn)$. Disjoint decomposition of $P$ into relatively open faces gives
$$ (\beta\times \omega)\cap{\rm Nor}\,P = \bigcup_{k=0}^{n-1}\bigcup_{F\in{\mathcal F}_k(P)}(\beta\cap {\rm relint}\, F)\times (\omega\cap\nu(P,F)).$$
By Lemma \ref{Lemma3.3}, $\Gamma(P,\cdot)$ is concentrated on ${\rm Nor}\,P$, hence
\begin{equation}\label{6} 
\Gamma(P,\beta\times \omega) = \sum_{k=0}^{n-1} \sum_{F\in{\cal F}_k(P)} \Gamma(P,(\beta\cap{\rm relint}\,F) \times (\omega\cap \nu(P,F))).
\end{equation}
Thus, it is sufficient to determine $\Gamma(P, \beta\times \omega)$ for the case where $\beta\subset {\rm relint}\,F$ and $\omega\subset \nu(P,F)$, for some face $F\in{\mathcal F}_k(P)$.

Therefore, we consider the following data: a number $k\in\{0,\dots,n-1\}$, a $k$-dimensional linear subspace $L\subset{\mathbb R}^n$,  a bounded Borel set $\beta \subset L$, a Borel set $\omega\subset{\mathbb S}^{n-1}\cap L^\perp$, a $k$-dimensional polytope $P\subset L$ with $\beta\subset{\rm relint}\,P$, and we determine $\Gamma(P,\beta\times \omega)$ in this case.

First, for fixed $\omega$, we  consider the functional
$$ \beta\mapsto \Gamma(P,\beta\times \omega)$$
for bounded Borel sets $\beta\subset L$, where $P$ is chosen such that $\beta\subset{\rm relint}\,P$ (the particular choice of  $P$ is irrelevant, since $\Gamma$ is locally defined).

Let $f(\beta)$ be a coordinate of $\Gamma(P,\beta\times\omega)$. We choose a polytope $P\subset L$ with $\beta\subset {\rm relint}\,P$. For $t\in L$ we then have $\beta+t\subset {\rm relint}(P+t)$ and it follows that $f(\beta+t)=f(\beta)$. Thus, $f$ is a translation invariant finite signed measure on ${\mathcal B}_b(L)$ and hence a constant multiple of Lebesgue measure in $L$, where the factor depends on $L$ and $\omega$. We conclude that
$$ \Gamma(P,\beta\times\omega)=a(L,\omega){\mathcal H}^k(\beta)$$
with a tensor $a(L,\omega)\in{\mathbb T}^p$. It is shown in \cite{Sch12} that $a(L,\cdot)$ is a ${\mathbb T}^p$-valued measure satisfying
$$ a(\vartheta L,\vartheta \omega) =\vartheta a(L,\omega)\qquad\mbox{for }\vartheta\in{\rm O}(n)$$
and $\vartheta a(L,\omega)= a(L,\omega)$ if $\vartheta$ fixes $L^\perp$ pointwise. From this, it is deduced in \cite{Sch12} (in particular, Lemmas 3 and 4) that
$$ a(L,\omega) = \sum_{j=0}^{\lfloor p/2\rfloor} Q_L^j \sum_{i=0}^{\lfloor p/2\rfloor} c_{pkij}Q_{L^\perp}^i\int_\omega u^{p-2j-2i}\,{\mathcal H}^{n-k-1}(\D u) $$
with real constants $c_{pkij}$ (the dependence of the coefficients on $k$ was not shown explicitly in \cite{Sch12}) and 
$c_{p0ij}=0$ for $j\ge 1$. It follows that
$$ \Gamma(P,\beta\times\omega)= \sum_{i,j=0}^{\lfloor p/2\rfloor}c_{pkij} Q_L^j (Q-Q_L)^i \int_\beta {\mathcal H}^k(\D x)\int_\omega u^{p-2j-2i}\,{\mathcal H}^{n-k-1}(\D u). $$
Now let $P\in{\mathcal P}^n$, $\beta\in {\mathcal B}({\mathbb R}^n)$ and $\omega\in {\mathcal B}({\mathbb S}^{n-1})$ be arbitrary. From (\ref{6}) we get
\begin{eqnarray*} 
\Gamma(P,\beta\times \omega) &=& \sum_{k=0}^{n-1} \sum_{F\in{\cal F}_k(P)} \Gamma(P,(\beta\cap{\rm relint}\,F)\times(\omega\cap \nu(P,F)))\\
& =& \sum_{k=0}^{n-1} \sum_{F\in{\cal F}_k(P)} \sum_{i,j=0}^{\lfloor p/2\rfloor}c_{pkij} Q_{L(F)}^j (Q-Q_{L(F)})^i\\
& & \times\int_{\beta\cap F} {\mathcal H}^k(\D x)\int_{\omega\cap \nu(P,F)} u^{p-2j-2i}\,{\mathcal H}^{n-k-1}(\D u).
\end{eqnarray*}
This extends to
\begin{eqnarray*} 
\Gamma(P,\eta) &= & \sum_{i,j=0}^{\lfloor p/2\rfloor} \sum_{k=0}^{n-1}c_{pkij} \sum_{F\in{\cal F}_k(P)}  Q_{L(F)}^j (Q-Q_{L(F)})^i \\
& &\times \int_F \int_{\nu(P,F)} {\bf 1}_\eta(x,u) u^{p-2j-2i}\,\,{\mathcal H}^{n-k-1}(\D u)\,{\mathcal H}^k(\D x)
\end{eqnarray*}
for all $\eta\in{\mathcal B}(\Sigma^n)$. Thus, $\Gamma$ is a linear combination, with constant coefficients, of the mappings $Q^m\phi^{0,s,j}_k$, with $m,s,j,k\in{\mathbb N}_0$, $k\le n-1$ and $2m+2j+s=p$, but $j=0$ if $k\in\{0,n-1\}$. 
\end{proof}

Now we indicate how Theorem \ref{Theorem2.2} can be derived from Theorem \ref{Theorem3.2}. Since we already know that the mappings $Q^m\phi_k^{p,s,j}$ have the properties (a), (b), (c) of Theorem \ref{Theorem2.2} and since linear independence follows from Theorem \ref{Theorem3.1}, it remains to show that a mapping $\Gamma$ with the properties in Theorem \ref{Theorem2.2} is a linear combination of mappings $Q^m\phi_k^{p,s,j}$. For this, we extend an argument used by Alesker \cite{Ale99b}. By (\ref{1}) we have
$$ \Gamma(P+t,\eta+t)= \sum_{i=0}^p \Gamma_{p-i}(P,\eta)\frac{t^i}{i!}.$$
By Lemma \ref{Lemma3.1},  $\Gamma_0$ is translation invariant. Hence, $\Gamma_0$ has all the properties of $\Gamma$ in Theorem \ref{Theorem3.2} (where we have to put $p=0$). It follows that
$$ \Gamma_0 =\sum_{m,s,j,k \atop 2m+2j+s=0} c_{msjk} Q^m\phi_k^{0,s,j}$$
with real constants $ c_{msjk}$. Precisely as in the proof of (\ref{2.7}) we get
\begin{equation}\label{n3.1} 
\phi_k^{r,s,j}(P+t,\eta+t) =\sum_{i=0}^r \phi_k^{r-i,s,j}(P,\eta)\frac{t^i}{i!}. 
\end{equation}
Hence, if we define
$$ \Delta:= \sum_{m,s,j,k \atop 2m+2j+s=0} c_{msjk}  Q^m\phi_k^{p,s,j},$$
then 
$$ \Delta(P+t,\eta+t)= \sum_{i=0}^p \Delta_{p-i}(P,\eta)\frac{t^i}{i!}$$
with tensors $\Delta_{p-j}(P,\eta)\in{\mathbb T}^{p-j}$, and here $\Delta_0=\Gamma_0$. Therefore, the mapping $\Gamma':=\Gamma-\Delta$ satisfies
$$ \Gamma'(P+t,\eta+t) = \sum_{j=0}^{p-1} \Gamma'_{p-j}(P,\eta)\frac{t^j}{j!}.$$
By Lemma \ref{Lemma3.1}, $\Gamma'_1$ is translation invariant. Hence, $\Gamma'_1$ has all the properties of $\Gamma$ in Theorem \ref{Theorem3.2} (where now we have to put $p=1$). Therefore,
$$ \Gamma'_1 =\sum_{m,s,j,k \atop 2m+2j+s=1} c'_{msjk} Q^m\phi_k^{0,s,j}$$
with real constants $c'_{msjk}$. Subtracting from $\Gamma'$ a suitable linear combination of tensor valuations $Q^m\phi_k^{p-1,s,j}$, we obtain a mapping $\Gamma''$ which is translation covariant of degree $p-2$. We can now repeat the argument, apply Theorem \ref{Theorem3.2} with $p=2$, and so on. After finitely many steps, $\Gamma$ is represented as a linear combination of mappings $Q^m\phi_k^{p,s,j}$, thus Theorem \ref{Theorem2.2} is proved. \qed

The local Minkowski tensors appearing in Theorem \ref{Theorem2.2} are valuations, though this is not one of the assumptions of the theorem. The valuation property was asserted in \cite{Sch12}. We give a proof.

\begin{theorem}\label{T3.3}
For each $\eta\in{\mathcal B}(\Sigma^n)$, the mapping $\phi_k^{r,s,j}(\cdot,\eta)$ is a valuation on $\Pn$.
\end{theorem}

\begin{proof}
To show that a function $\varphi$ on $\Pn$ with values in an abelian group is a valuation, it suffices to show that, after supplementing the definition by $\varphi(\emptyset)=0$, one has
\begin{equation}\label{A} 
\varphi(P\cap H^-) + \varphi(P\cap H^+) = \varphi(P)+\varphi(P\cap H)
\end{equation}
for $P\in\Pn$ and every hyperplane $H$, where $H^-,H^+$ are the two closed halfspaces bounded by $H$. This was first noticed by Sallee \cite{Sal68}. Hence, using the representation (\ref{2.6}) and the abbreviation
$$  J(P,F):= C_{n,k}^{r,s}\, Q_{L(F)}^{\hspace*{1pt}j}\int_F \int_{\nu(P,F)} {\bf 1}_\eta(x,u) x^r u^s\, \Ha^{n-k-1}(\D u)\,\Ha^k(\D x),$$
we have to show that
\begin{eqnarray}\label{B}
& & \sum_{F\in\F_k(P\cap H^-)} J(P\cap H^-,F) \enspace+ \sum_{F\in\F_k(P\cap H^+)} J(P\cap H^+,F) \nonumber\\
& & = \sum_{F\in\F_k(P)} J(P,F) \enspace+ \sum_{F\in\F_k(P\cap H)} J(P\cap H,F).
\end{eqnarray}

Let a polytope $P\in \Pn$ and a hyperplane $H$ be given. We consider a face 
$$ F\in {\mathcal F}_k(P)\cup {\mathcal F}_k (P\cap H)$$
and distinguish the following five cases.

Case 1: $F\not\subset H$ and  $F\subset H^-$. Then 
$$ F\in {\mathcal F}_k(P) \cap {\mathcal F}_k(P\cap H^-), \qquad F\notin{\mathcal F}_k(P\cap H^+)\cup {\mathcal F}_k(P\cap H).$$ 
Since $\nu(P\cap H^-,F)=\nu(P,F)$, it follows that $J(P\cap H^-,F)=J(P,F)$.

Case 2: $F\not\subset H$ and $F\subset H^+$. Similarly as in Case 1, it follows that $J(P\cap H^+,F)=J(P,F)$.

Case 3: $F\not\subset H$, $F\not \subset H^-$, and $F\not \subset H^+$. Then $F^-:= F\cap H^-\in{\mathcal F}_k(P\cap H^-)$ and $F^+:= F\cap H^+\in{\mathcal F}_k(P\cap H^+)$. Moreover, $L(F^-)=L(F^+)=L(F)$, $\nu(P\cap H^-,F^-)=\nu(P\cap H^+,F^+)=\nu(P,F)$,  $F^- \cup F^+= F$, and ${\mathcal H}^k(F^-\cap F^+)=0$. It follows that $J(P\cap H^-,F^-) + J(P\cap H^+,F^+) = J(P,F)$. 

Case 4: $F\subset H$ and $F\notin{\mathcal F}_k(P)$. Then there is a unique face $G\in{\mathcal F}_{k+1}(P)$ such that $F= G\cap H$. We have
$$ F\in {\mathcal F}_k(P\cap H^-)\cap {\mathcal F}_k(P\cap H^+)\cap {\mathcal F}_k(P\cap H).$$
We choose a point $q$ in the relative interior of $F$. For $Q\in\{P\cap H^-,P\cap H^+,P\cap H\}$ and for $u\in {\mathbb S}^{n-1}$ we then have ${\bf 1}_{\nu(Q,F)}(u)= j(Q,q,q+u)$, where $j$ denotes the index function defined in \cite[p. 231]{Sch14}; further,  ${\bf 1}_{\nu(P,G)}(u)= j(P,q,q+u)$. From the additivity of $j(\cdot,q,q+u)$ ({\em loc.
cit.}) we obtain
$$ {\bf 1}_{\nu(P\cap H^-,F)} + {\bf 1}_{\nu(P\cap H^+,F)} = {\bf 1}_{\nu(P,G)} + {\bf 1}_{\nu(P\cap H,F)}.$$
Since ${\mathcal H}^{n-k-1}(\nu(P,G))=0$, it follows that $J(P\cap H^-,F)+J(P\cap H^+,F) =J(P\cap H,F)$.

Case 5: $F\subset H$ and $F\in{\mathcal F}_k(P)$. Then (irrespective of whether $H$ supports $P$ or not) we have
$$ F\in {\mathcal F}_k(P\cap H^-)\cap {\mathcal F}_k(P\cap H^+)\cap {\mathcal F}_k(P\cap H)$$
and, using the index function similarly as in Case 4,
$$ {\bf 1}_{\nu(P\cap H^-,F)} + {\bf 1}_{\nu(P\cap H^+,F)} = {\bf 1}_{\nu(P,F)} + {\bf 1}_{\nu(P\cap H,F)}.$$
This gives $J(P\cap H^-,F)+J(P\cap H^+,F) = J(P,F)+J(P\cap H,F)$.

In the five cases, each of the pairs $(P,F)$ with $F\in{\mathcal F}_k(P)$ and $(P\cap H,F)$ with $F\in{\mathcal F}_k(P\cap H)$ was considered precisely once. But also each of the pairs $(P\cap H^-,F)$ with $F\in{\mathcal F}_k(P\cap H^-)$ and $(P\cap H^+,F)$ with $F\in{\mathcal F}_k(P\cap H^+)$ was obtained precisely once (in Case 3 with the notation $F^-$ or $F^+$). Hence, if we add all the established equations for $J$, we obtain (\ref{B}).
\end{proof}

We return to general convex bodies and prepare the proof of Theorem \ref{Theorem2.3} in Section \ref{sec5} by the final lemma of this section.
We say that $\Gamma:\Kn\times\B(\Sigma^n)\to\T^p$ is {\em homogeneous of degree} $k$ if $\Gamma(\lambda K,\lambda \eta) = \lambda^k\Gamma(K,\eta)$ for all $K\in\Kn$, $\eta\in\B(\Sigma^n)$ and $\lambda>0$, where $\lambda\eta:=\{(\lambda x,u):(x,u)\in\eta\}$.

\vspace{2mm}

\begin{lemma}\label{Lemma3.5}
Let $p\in\N_0$. Let $\Gamma:\Kn\times\B(\Sigma^n)\to\T^p$ be a mapping with the following properties.\\
$\rm (a)$ $\Gamma(K,\cdot)$ is a $\T^p$-valued measure, for each $K\in\Kn$;\\ 
$\rm (b)$ $\Gamma$ is translation invariant and rotation covariant;\\
$\rm (c)$ $\Gamma$ is locally defined;\\
$\rm (d)$ $\Gamma$ is weakly continuous.\\
Then $\Gamma= \sum_{k=0}^{n-1} \Gamma_k$, where each $\Gamma_k$ has properties $\rm (a) - (d)$ and is homogeneous of degree $k$.
\end{lemma}

\begin{proof} To the restriction $\Gamma'$ of $\Gamma$ to ${\mathcal P}^n\times{\mathcal B}(\Sigma^n)$ we can apply Theorem \ref{Theorem3.2} and deduce that $\Gamma'$ is a linear combination of the mappings $Q^m\phi_k^{0,s,j}$, $0\le k\le n-1$ and $m,s,j\in{\mathbb N}_0$. Each $Q^m\phi_k^{0,s,j}$ has properties $\rm (a) - (c)$ and is homogeneous of degree $k$. Hence, we can write $\Gamma'= \sum_{k=0}^{n-1} \Gamma_k$, where each $\Gamma_k:{\mathcal P}^n\times{\mathcal B}(\Sigma^n)\to{\mathbb T}^p$ has properties $\rm (a) - (c)$ and is homogeneous of degree $k$.

We argue similarly as in the proof of Lemma \ref{Lemma3.2}. For $P\in{\mathcal P}^n$, $\eta\in{\mathcal B}(\Sigma^n)$ and $\lambda>0$ we have
\begin{equation}\label{10} 
\Gamma(\lambda P,\lambda \eta)=\sum_{k=0}^{n-1} \lambda^k\Gamma_k(P,\eta).
\end{equation}
For fixed $P$ and $\eta$, let $f(\lambda)$ be a coordinate of $\Gamma(\lambda P,\lambda\eta)$, and let $f_k$ be the corresponding coordinate of $\Gamma_k(P,\eta)$. Then $f(\lambda)=\sum_{k=0}^{n-1} \lambda^k f_k$. We insert $\lambda =1,\dots,n$ and solve the resulting system of linear equations for the $f_k$, to obtain $f_k=\sum_{q=1}^n b_{kq}f(q)$, $k=0,\dots,n-1$, with constants $b_{kq}$ depending only on $n,k,q$. Thus,
\begin{equation}\label{8} 
\Gamma_k(P,\eta)=\sum_{q=1}^n b_{kq} \Gamma(qP,q\eta),\quad k=0,\dots,n-1.
\end{equation}
This holds for arbitrary $P\in{\mathcal P}^n$ and $\eta\in{\mathcal B}(\Sigma^n)$. Now we extend the definition of $\Gamma_k$ by
\begin{equation}\label{9} 
\Gamma_k(K,\eta):=\sum_{q=1}^n b_{kq} \Gamma(qK,q\eta),\quad k=0,\dots,n-1,
\end{equation}
for $K\in{\mathcal K}^n$ and $\eta\in{\mathcal B}(\Sigma^n)$. Then $\Gamma_k$ has properties $\rm (a) - (d)$.

For given $K\in{\mathcal K}^n$, let $(P_i)_{i\in{\mathbb N}}$ be a sequence of polytopes converging to $K$. Since $\Gamma(q P_i,q\,\cdot) \xrightarrow{w} \Gamma(qK,q\,\cdot)$ (where $\xrightarrow{w}$ denotes weak convergence), it follows from (\ref{8}) and (\ref{9}) that $\Gamma_k(P_i,\cdot)  \xrightarrow{w} \Gamma_k(K,\cdot)$. This implies that the extended $\Gamma_k$ is homogeneous of degree $k$. From (\ref{10}) and weak continuity it follows that
$$ \Gamma(\lambda K,\lambda \eta)=\sum_{k=0}^{n-1} \lambda^k\Gamma_k(K,\eta).$$
Thus, $\Gamma=\sum_{k=0}^{n-1} \Gamma_k$, where each $\Gamma_k$ has the properties $\rm (a) - (d)$ and is homogeneous of degree $k$. 
\end{proof}

\section{Weakly continuous extensions}\label{sec4}

So far, we have defined the generalized local Minkowski tensors $\phi_k^{r,s,j}$ with $j\ge 1$ only for polytopes. Lemma \ref{Lemma3.4} allows us to define $\phi_{n-1}^{r,s,j}$, for general convex bodies $K\in\Kn$, by
$$\phi_{n-1}^{r,s,j}(K,\cdot) = \sum_{i=0}^j (-1)^i\binom{j}{i} \frac{(s+2i)!\omega_{1+s+2i}}{s!\omega_{1+s}} Q^{j-i}\phi_{n-1}^{r,s+2i}(K,\cdot).$$
Then $\phi_{n-1}^{r,s,j}$ is weakly continuous, since this holds for $\phi_{n-1}^{r,l}$. For $1\le k\le n-2$ and $j\ge 2$, we shall see in the next section that no weakly continuous extension of $\phi_k^{r,s,j}$ to $\Kn$ is possible. 

The remaining case, $j=1$, is the subject of this section. We assume that $1\le k\le n-2$ and construct a weakly continuous extension of  $\phi_k^{r,s,1}$ from $\Pn$ to $\Kn$.

For this, we need some basic terminology and results of geometric measure theory, for which we refer to Federer's book \cite{Fed69}. We start by describing support measures in terms of currents. This point of view was first suggested and explored (for sets of positive reach) by Z\"ahle \cite{Zae86}. We shall need the normal cycle associated with a convex body. To introduce it, we  remark that $\Nor K\subset\R^n\times\R^n=\R^{2n}$, the normal bundle of $K$, is an $(n-1)$-rectifiable set. For $\mathcal{H}^{n-1}$-almost all $(x,u)\in\Nor K$, the set of $(\mathcal{H}^{n-1}\fed \Nor K,n-1)$ approximate tangent vectors at $(x,u)$ is an $(n-1)$-dimensional linear subspace of $\RR^{2n}$, which is denoted by 
$\textrm{Tan}^{n-1}(\mathcal{H}^{n-1}\fed \Nor K,(x,u))$. 
This {\em approximate tangent space} is spanned by the orthonormal basis $(a_1(x,u),\dots,a_{n-1}(x,u))$,  where   
$$
a_i(x,u):=\left( \frac{1}{\sqrt{1+{k_i(x,u)}^2}}\,b_i(x,u),\frac{k_i(x,u)}{\sqrt{1+{k_i(x,u)}^2}}\,b_i(x,u)\right).
$$
Here, $(b_1(x,u),\ldots,b_{n-1}(x,u))$ is a suitable orthonormal basis of $u^\perp$ (the orthogonal complement of the linear subspace spanned by $u$), which is chosen so that the basis  $(b_1(x,u),\ldots,b_{n-1}(x,u),u)$ has the same orientation as the standard basis $(e_1,\ldots,e_n)$ of $\R^n$, and $k_i(x,u)\in [0,\infty]$ for $i=1,\ldots,n-1$, with the understanding that
$$  \frac{1}{\sqrt{1+{k_i(x,u)}^2}}=0\quad\mbox{and}\quad\frac{k_i(x,u)}{\sqrt{1+{k_i(x,u)}^2}}=1 \qquad\mbox{if }k_i(x,u)=\infty.$$
Note that the dependence of $a_i,b_i,k_i$ on $K$ is not made explicit by our notation. The data $b_i, k_i$, $i=1,\ldots,n-1$, are essentially uniquely determined (cf.~\cite[Prop.~3 and Lem.~2]{RZ05}). Moreover, we can assume that $b_i(x+\varepsilon u,u)=b_i(x,u)$, independent of $\varepsilon>0$, where $(x,u)\in\Nor K$ and $(x+\varepsilon u,u)\in\Nor K_\varepsilon$  with $K_\varepsilon:=K+\varepsilon B^n$. See \cite{Zae86, RZ01, RZ05, Hug95, Hug98}  for a geometric interpretation of the numbers $k_i(x,u)$ as generalized curvatures and for arguments establishing these facts. Hence, for $\mathcal{H}^{n-1}$-almost all $(x,u)\in\Nor K$, we 
can define an $(n-1)$-vector
$$
a_K(x,u):=a_1(x,u)\wedge\ldots\wedge a_{n-1}(x,u),
$$
which defines an orientation of $\textrm{Tan}^{n-1}(\mathcal{H}^{n-1}\fed \Nor(K),(x,u))$. Then an $(n-1)$-dimensional current in 
$\R^{2n}$ is defined by 
$$
T_K:=\left(\mathcal{H}^{n-1}\fed\Nor K\right)\wedge a_K,
$$
which is known as the {\em normal cycle} of $K$. More explicitly, 
$$
T_K(\varphi)=\int_{\Nor K}\langle a_K(x,u),\varphi(x,u)\rangle\, \mathcal{H}^{n-1}(d(x,u)),
$$
for all $\mathcal{H}^{n-1}\fed \Nor K$-integrable functions $\varphi:\R^{2n}\to \bigwedge^{n-1}\R^{2n}$, where we write $\langle\cdot\,,\cdot\rangle$ for the pairing of $m$-vectors and $m$-covectors, as in \cite[p.~17]{Fed69} (but we continue to use $\langle\cdot\,,\cdot\rangle$ also for the scalar product in $\R^n$, which cannot lead to ambiguities). Here we use that $T_K$ is a rectifiable current, which has compact support, and thus $T_K$ can be defined for a larger class of functions than just for the class of smooth differential forms. 

In order to define the Lipschitz--Killing forms $\varphi_k$, $k\in\{0,\ldots,n-1\}$, we need the projection maps $\Pi_1:\R^n\times\R^n\to\R^n$, $(x,u)\mapsto x$, and $\Pi_2:\R^n\times\R^n\to\R^n$, $(x,u)\mapsto u$. Let $\Omega_n$ be  the volume form on $\R^n$ with the orientation chosen so that 
$$
\Omega_n(e_1,\ldots,e_n)=\langle e_1\wedge\ldots\wedge e_n,\Omega_n\rangle=1,
$$ 
where $(e_1,\dots,e_n)$ is the standard orthonormal basis of ${\mathbb R}^n$. Then differential forms $\varphi_k:\R^{2n} \to \bigwedge^{n-1}\R^{2n}$, $k\in\{0,\ldots,n-1\}$,  of degree $n-1$ on $\R^{2n}$ are defined by 
\begin{eqnarray*}
& & \varphi_k(x,u)(\xi_1,\ldots,\xi_{n-1})\\
&  & := \frac{1}{k!(n-1-k)!\omega_{n-k}}\sum_{\sigma\in S(n-1)}\sgn(\sigma)\left\langle
\bigwedge_{i=1}^k\Pi_1\xi_{\sigma(i)}\wedge \bigwedge_{i=k+1}^{n-1}\Pi_2\xi_{\sigma(i)}\wedge u,\Omega_n\right\rangle,
\end{eqnarray*}
where $(x,u)\in \R^n\times\R^n=\R^{2n}$, $\xi_1,\ldots,\xi_{n-1}\in \R^{2n}$, and $S(n-1)$ denotes the set of all permutations of $\{1,\ldots,n-1\}$. Note that this definition is equivalent to the one given in \cite{Zae86}. A straightforward calculation yields that 
$$
\langle a_K(x,u),\varphi_k(x,u)\rangle
=\frac{1}{\omega_{n-k}}\sum_{|I|=n-1-k}\frac{\prod_{i\in I}k_i(x,u)}{\prod_{i=1}^{n-1}\sqrt{1+k_i(x,u)^2}}
$$
for $\mathcal{H}^{n-1}$-almost all $(x,u)\in\Nor K$. The summation extends over all subsets $I$ of $\{1,\ldots,n-1\}$ of cardinality 
$n-1-k$. We adopt the convention that a product over an empty set is defined as 1. Then, for $\eta\in\mathcal{B}(\Sigma^n)$ we obtain
$$
T_K\left(\mathbf{1}_\eta {\varphi}_k \right)=\Lambda_k(K,\eta),
$$
which provides a representation of the $k$th support measure of $K$ in terms of the normal cycle of $K$, 
evaluated at the $k$th Lipschitz--Killing form $\varphi_k$.

Taking this procedure as a model, we now introduce new tensor-valued differential forms of degree $n-1$. By evaluating these forms at the normal cycle, we shall obtain the requested continuous extension of $\phi_k^{r,s,1}$ to general convex bodies. Let $n\ge 3$, $k\in \{1,\ldots,n-2\}$ and $r,s\in \N_0$.  Let $(x,u)\in \R^n\times\R^n=\R^{2n}$ and $\mathbf{v}=(v_1,\ldots,v_{r+s+2})\in (\R^n)^{r+s+2}$. For $\xi_1,\ldots,\xi_{n-1}\in\R^{2n}$, we define
\begin{eqnarray*}
& & \widetilde{\varphi}_k^{r,s}(x,u;\mathbf{v};\xi_1,\ldots,\xi_{n-1}):=\frac{C^{r,s}_{n,k}}{(k-1)!(n-1-k)!}\, x^r(v_1,\ldots,v_r)u^s(v_{r+1},\ldots,v_{r+s})\\
& & \times\sum_{\sigma\in S(n-1)}\sgn(\sigma)\left\langle v_{r+s+1},\Pi_1\xi_{\sigma(1)}\right\rangle 
\left\langle v_{r+s+2}\wedge\bigwedge_{i=2}^k\Pi_1\xi_{\sigma(i)}\wedge\bigwedge_{i=k+1}^{n-1}\Pi_2\xi_{\sigma(i)}\wedge u,\Omega_n\right\rangle .
\end{eqnarray*}
We omit the wedge product $\bigwedge_{i=2}^k$ if $k=1$.  For fixed $x,u,\mathbf{v}$, the map
$$
\widetilde{\varphi}_k^{r,s}(x,u;\mathbf{v};\cdot):(\R^{2n})^{n-1}\to\R
$$
is multilinear and alternating, and therefore it is an element of  $\bigwedge^{n-1}\R^{2n}$. Next, we symmetrize by defining
\begin{eqnarray*}
& & {\varphi}_k^{r,s}(x,u;\mathbf{v};\xi_1,\ldots,\xi_{n-1})\\
& & :=\frac{1}{(r+s+2)!}\sum_{\tau\in S(r+s+2)}
\widetilde{\varphi}_k^{r,s}(x,u;v_{\tau(1)},\ldots,v_{\tau(r+s+2)};\xi_1,\ldots,\xi_{n-1}).
\end{eqnarray*} 
Then the mapping 
$$(\xi_1,\dots,\xi_{n-1})\mapsto \varphi_k^{r,s}(x,u;\,\cdot\,;\xi_1,\dots,\xi_{n-1}),$$ 
which we denote briefly by ${\varphi}_k^{r,s}(x,u)$, is an $(n-1)$-covector of $\mathbb{T}^{r+s+2}$, hence an element of $\bigwedge^{n-1}(\R^{2n},\mathbb{T}^{r+s+2})$. Therefore, the map
$$
{\varphi}_k^{r,s}:\R^{2n}\to \bigwedge\nolimits^{n-1}(\R^{2n},\mathbb{T}^{r+s+2}),\qquad (x,u)\mapsto {\varphi}_k^{r,s}(x,u),
$$
is a differential form of degree $n-1$ on $\R^{2n}$ with coefficients in $\mathbb{T}^{r+s+2}$ (see \cite[p.~351]{Fed69}), that is, an element of $\mathcal{E}^{n-1}(\R^{2n},\mathbb{T}^{r+s+2}):=\mathcal{E} (\R^{2n},\bigwedge^{n-1}(\R^{2n},\mathbb{T}^{r+s+2}))$. In particular,
$$
\langle a,{\varphi}_k^{r,s}(x,u)\rangle\in \mathbb{T}^{r+s+2}
$$
for all $(x,u)\in \R^{2n}$ and $a\in \bigwedge_{n-1}\R^{2n}$, where we use the linear isomorphism mentioned in \cite[p. 17]{Fed69} to identify $ \bigwedge^{n-1}(\R^{2n},W)$ and $ {\rm Hom}\left(\bigwedge_{n-1}\R^{2n},W\right)$, for an arbitrary vector space $W$. We remark that a straightforward calculation shows that
\begin{equation}\label{eqcov}
\langle \vartheta a,{\varphi}_k^{r,s}(\vartheta x,\vartheta u)\rangle=\vartheta \langle  a,{\varphi}_k^{r,s}( x, u)\rangle,
\end{equation} 
for all $\vartheta\in {\rm O}(n)$, where in each case the natural operation of the rotation group is used (in particular, $\vartheta \xi:=(\vartheta p,\vartheta q)$ for $\xi=(p,q)\in\R^n\times\R^n =\R^{2n}$).

The next lemma shows that the differential forms $\varphi_k^{r,s}$ serve the intended purpose.

\begin{lemma}\label{lemma4.2} 
If $P\in\mathcal{P}^n$ and $\eta\in{\mathcal B}(\Sigma^n)$, then 
$$
T_P\left(\mathbf{1}_\eta {\varphi}_k^{r,s} \right)=\phi_k^{r,s,1}(P,\eta).
$$
\end{lemma}

\begin{proof} For given $P\in\mathcal{P}^n$ and $\eta\in{\mathcal B}(\Sigma^n)$, we have to show that
\begin{eqnarray*}
T_P\left(\mathbf{1}_{\eta} \widetilde{\varphi}_k^{r,s}(\cdot\,;\mathbf{v};\cdot) \right)&=&C^{r,s}_{n,k}\sum_{F\in\mathcal{F}_k(P)} Q_{L(F)}(v_{r+s+1},v_{r+s+2})\\
& & \times\,  \int_{\eta\cap(F\times\nu(P,F))} x^r(v_1,\ldots,v_r)u^s(v_{r+1},\ldots,v_{r+s})
\,\mathcal{H}^{n-1}(\D(x,u)),
\end{eqnarray*}
for all $\mathbf{v}=(v_1,\ldots,v_{r+s+2})\in (\R^n)^{r+s+2}$. Subsequent symmetrization then yields the result, in view of (\ref{2.6}).

To see this, we use the disjoint decomposition
$$
\eta\cap \Nor P=\bigcup_{j=0}^{n-1}\bigcup_{F\in\mathcal{F}_j(P)} \eta\cap(\textrm{relint }F\times\nu(P,F))
$$
to get
\begin{eqnarray*}
& & T_P\left(\mathbf{1}_{\eta} \widetilde{\varphi}_k^{r,s}(\cdot\,;\mathbf{v};\cdot)\right)\\
& & =\int_{\eta\cap\Nor P} \left\langle a_P(x,u), \widetilde{\varphi}_k^{r,s}(x,u;\mathbf{v};\cdot)\right\rangle  \mathcal{H}^{n-1}(\D (x,u)) \allowdisplaybreaks\\
& & = \frac{C^{r,s}_{n,k}}{(k-1)!(n-1-k)!}\sum_{j=0}^{n-1} \sum_{F\in\mathcal{F}_j(P)} \int_{\eta\cap(F\times \nu(P,F)) }
x^r(v_1,\ldots,v_r)u^s(v_{r+1},\ldots,v_{r+s})\\
& & \hspace{4mm} \times \sum_{\sigma \in S(n-1)}\sgn(\sigma)\frac{\prod_{i=k+1}^{n-1}k_{\sigma(i)}(x,u)} {\prod_{i=1}^{n-1}\sqrt{1+k_i(x,u)^2}}\left\langle v_{r+s+1},b_{\sigma(1)}(x,u)\right\rangle \\
& & \hspace{4mm}\times \left\langle v_{r+s+2}\wedge\bigwedge_{i=2}^{n-1}b_{\sigma(i)}(x,u)\wedge u,\Omega_n\right\rangle\mathcal{H}^{n-1}(\D(x,u)) \allowdisplaybreaks\\
& & =\frac{C^{r,s}_{n,k}}{(k-1)!(n-1-k)!}\sum_{j=0}^{n-1}
\sum_{F\in\mathcal{F}_j(P)} \int_{\eta\cap(F\times \nu(P,F)) }x^r(v_1,\ldots,v_r)u^s(v_{r+1},\ldots,v_{r+s})\\
& &\hspace{4mm}\times \sum_{\sigma\in S(n-1)}\frac{\prod_{i=k+1}^{n-1}k_{\sigma(i)}(x,u)} {\prod_{i=1}^{n-1} \sqrt{1+k_i(x,u)^2}}\, b_{\sigma(1)}(x,u)^2(v_{r+s+1},v_{r+s+2})\,\mathcal{H}^{n-1}(\D(x,u)).
\end{eqnarray*}
In the final step we have used that $(b_1(x,u),\dots,b_{n-1}(x,u),u)$ is an orthonormal basis of $\R^n$ with the same orientation as the standard basis and hence
$$ \left\langle v_{r+s+2}\wedge\bigwedge_{i=2}^{n-1}b_{\sigma(i)}(x,u)\wedge u,\Omega_n\right\rangle =\left\langle v_{r+s+2},b_{\sigma(1)}(x,u)\right\rangle\sgn(\sigma).$$
If $F\in\mathcal{F}_j(P)$, then, for $\mathcal{H}^{n-1}$-almost all $(x,u)\in F\times\nu(P,F)$,  exactly $j$ 
of the numbers $k_i(x,u)$ are zero and $n-1-j$ of these numbers are infinite. Moreover, in this situation, $k_i(x,u)=0$ if and only if 
the corresponding vector $b_i(x,u)$ is in $L(F)$. Hence, if $j\neq k$, then 
$$
\frac{\prod_{i=k+1}^{n-1}k_{\sigma(i)}(x,u)}{\prod_{i=1}^{n-1}\sqrt{1+k_i(x,u)^2}}=0\quad \textrm{for all }\sigma\in S(n-1).
$$
In fact, if $j>k$, then the numerator is zero, and for $j<k$ the number of indices  $i\in\{1,\ldots,n-1\}$ such that $k_i(x,u)=\infty$ in the numerator is smaller than the corresponding number of indices in the denominator. If $F\in{\mathcal F}_k(P)$, $(x,u)\in F\times\nu(P,F)$ and, say, $b_1(x,u),\dots,b_k(x,u)\in L(F)$, then
$$ Q_{L(F)}=b_1(x,u)^2+\dots+ b_k(x,u)^2.$$
Hence, we conclude that
\begin{align*}
&T_P\left(\mathbf{1}_{\eta} \widetilde{\varphi}_k^{r,s}(\cdot\,;\mathbf{v};\cdot)\right)\\
&=\frac{C^{r,s}_{n,k}}{(k-1)!(n-1-k)!} \sum_{F\in\mathcal{F}_k(P)} (k-1)!(n-1-k)!Q_{L(F)}(v_{r+s+1},v_{r+s+2})\\
&\hspace{4mm}\times \int_{\eta\cap(F\times \nu(P,F))}x^r(v_1,\ldots,v_r)u^s(v_{r+1},\ldots,v_{r+s})\,\mathcal{H}^{n-1} (\D(x,u)),
\end{align*}
as stated.
\end{proof}

\begin{lemma}\label{Lemma4.1}
{\rm (a)} For $K\in\Kn$, $T_K$ is a cycle.\\
{\rm (b)} The map $K\mapsto T_K$ is a valuation on $\Kn$.\\
{\rm (c)} If $K_i,K\in\Kn$, $i\in\N$, and $ K_i\to K$ in the Hausdorff metric, as $i\to\infty$, then $T_{K_i}\to T_K$ in the dual flat seminorm  for currents. 
\end{lemma}

\begin{proof} These facts are provided, for instance, in the more general context of sets with positive reach as Proposition 2.6 (assertion (a)), Theorem 2.2 (assertion (b)) and Theorem 3.1 (assertion (c)) in \cite{RZ01}.
\end{proof}

We remark that a strengthened form of the continuity assertion (c) of Lemma \ref{Lemma4.1}, namely local H\"older continuity of the normal cycles of convex bodies with respect to the Hausdorff metric and the dual flat seminorm, is proved in \cite{HS13}.

The third statement of Lemma \ref{Lemma4.1} implies that if $f:\R^{2n}\to\R$ is of class $C^\infty$, then the map
$$
\Kn\to\R,\qquad K\mapsto T_K\left(f {\varphi}_k^{r,s}  \right),
$$
is continuous. But then the same is true if $f$ is merely continuous, and thus $(K,\eta)\mapsto T_K\left(\mathbf{1}_\eta {\varphi}_k^{r,s} \right)$ is the weakly continuous 
extension of $(P,\eta)\mapsto \phi_k^{r,s,1}(P,\eta)$ from polytopes $P$ to general 
convex bodies. 

\begin{theorem}\label{Theorem4.1} 
The map $\Kn\times\B(\Sigma^n)\to\T^{r+s+2}$ defined by $(K,\eta)\mapsto T_K\left(\mathbf{1}_\eta \varphi_k^{r,s}\right)$, satisfies the properties {\rm (a) -- (d)} listed in Theorem $\ref{Theorem2.3}$.
\end{theorem}

\begin{proof} (a) Since $\Nor K$ has finite $(n-1)$-dimensional Hausdorff measure, the measure property follows from 
the dominated convergence theorem. (b) is implied since $\phi_k^{r,s,1}$ is isometry covariant on polytopes and this property is preserved under weak convergence. ((b) can also be shown directly. The rotation covariance follows, for instance, from \eqref{eqcov}.) (d) is implied by Lemma \ref{Lemma4.1} (c). It remains to verify that the tensor-valued measure is locally defined. For this, let $K,K'\in\Kn$ and $\eta\in\B(\Sigma^n)$ be such that $\eta\cap\Nor K=\eta\cap\Nor K'$. With $\Nor K$ and $\Nor K'$ also $\eta\cap\Nor K\cap \Nor K'$ is $(n-1)$-rectifiable. Therefore, for $\mathcal{H}^{n-1}$-almost all $(x,u)\in  \eta\cap\Nor K\cap \Nor K'$ the approximate tangent space of this intersection coincides with 
the one of $\Nor K$ and the one of $\Nor K'$ at $(x,u)$, and the orientations coincide. Hence we have $a_K(x,u)=a_{K'}(x,u)$ for $\mathcal{H}^{n-1}$-almost all $(x,u)\in \eta\cap\Nor K=\eta\cap\Nor K'$, which yields the assertion. 
\end{proof}

\begin{corollary}\label{Corollary4.1}
Let $r,s\in\N_0$ and $k\in\{1,\ldots,n-2\}$. Then, for each $\eta\in
\mathcal{B}(\Sigma^n)$, the map $K\mapsto \phi_k^{r,s,1}(K,\eta)$
is additive (a valuation) and Borel measurable on $\Kn$.
\end{corollary}

\begin{proof}
Since $K\mapsto T_K$ is a valuation on $\Kn$ by Lemma \ref{Lemma4.1} (b),
the first assertion follows from
$\phi_k^{r,s,1}(K,\eta)=T_K\left(\mathbf{1}_\eta {\varphi}_k^{r,s} \right)$
for $K\in\Kn$.

Let $f:\Sigma^n\to\R$ be a continuous function with compact support. Then
$$
K\mapsto \int_{\Sigma^n} f(x,u)\, \phi_k^{r,s,1}(K,\D (x,u)),\qquad K\in\Kn,
$$
is continuous by Lemma \ref{Lemma4.1} (c), and therefore measurable. The
second assertion is then implied by \cite[Lem.~12.1.1]{SW08}.
\end{proof}

\begin{remark}\label{Remarksec4} {\rm Since the global functionals $\phi_k^{r,s,1}(P,\Sigma^n)$ are continuous, Alesker's characterization theorem must yield  a representation for them. Such a representation was explicitly known before. In fact, for $r=0$ it follows from a relation of McMullen \cite[p.~269]{McM97} (see also \cite[Lem.~3.3]{HSS08b}) that
$$ \phi_k^{0,s,1}(P,\Sigma^n) = Q\Phi_k^{0,s}(P) -2\pi (s+2)\, \Phi_k^{0,s+2}(P).$$
The general case is provided by the second and the third displayed formula in \cite[p.~505]{HSS08b}.}
\end{remark}

Finally in this section, we express the new local tensor valuations $\phi^{r,s,1}_k(K,\cdot)$ for a general convex body $K$ in terms of the generalized curvatures $k_i(x,u)$ and the corresponding principal directions of curvature, given by the unit vectors $b_i(x,u)$, $i=1,\ldots, n-1$ (the result is relation (\ref{nn}) below). We put   
$$
\mathbb{K}(x,u):=\prod_{i=1}^{n-1}\sqrt{1+k_i(x,u)^2}.
$$ 
Similarly as in the beginning of the proof of Lemma \ref{lemma4.2} and with the notations introduced there, we obtain that
\begin{eqnarray*}
& & T_K\left(\mathbf{1}_{\eta} \widetilde{\varphi}_k^{r,s}(\cdot\,;\mathbf{v};\cdot)\right)\\
& & =\int_{\eta\cap\Nor K} \left\langle a_K(x,u), \widetilde{\varphi}_k^{r,s}(x,u;\mathbf{v};\cdot)\right\rangle  \mathcal{H}^{n-1}(\D (x,u))\\ %\allowdisplaybreaks\\
& & = \frac{C^{r,s}_{n,k}}{(k-1)!(n-1-k)!}\int_{\eta\cap\Nor K}x^r(v_1,\ldots,v_r)u^s(v_{r+1},\ldots,v_{r+s}) \\
& & \hspace{4mm} \times \sum_{\sigma \in S(n-1)}\sgn(\sigma)\frac{\prod_{i=k+1}^{n-1}k_{\sigma(i)}(x,u)} {\mathbb{K}(x,u)}\left\langle v_{r+s+1},b_{\sigma(1)}(x,u)\right\rangle \\
& & \hspace{4mm}\times \left\langle v_{r+s+2},b_{\sigma(1)}(x,u)\right\rangle\sgn(\sigma)\, \mathcal{H}^{n-1}(\D(x,u))
\allowdisplaybreaks\\
& & = \frac{C^{r,s}_{n,k}}{(k-1)!(n-1-k)!}\int_{\eta\cap\Nor K}x^r(v_1,\ldots,v_r)u^s(v_{r+1},\ldots,v_{r+s}) \\
& & \hspace{4mm}\times \sum_{\sigma \in S(n-1)}b_{\sigma(1)}(x,u)^2(v_{r+s+1},v_{r+s+2})\frac{\prod_{i=k+1}^{n-1} k_{\sigma(i)}(x,u)}{\mathbb{K}(x,u)}\,\mathcal{H}^{n-1}(\D(x,u)) \allowdisplaybreaks\\
& & = C^{r,s}_{n,k}\int_{\eta\cap\Nor K}x^r(v_1,\ldots,v_r)u^s(v_{r+1},\ldots,v_{r+s})\\
& & \hspace{4mm}\times \sum_{i=1}^{n-1}b_{i}(x,u)^2(v_{r+s+1},v_{r+s+2})\sum_{|I|=n-1-k \atop i\notin I} \frac{\prod_{j\in I} k_{j}(x,u)}{\mathbb{K}(x,u)}\,\mathcal{H}^{n-1}(\D(x,u)).
\end{eqnarray*}
From this we deduce that
\begin{eqnarray}\label{nn}
& & \phi^{r,s,1}_k(K,\eta)\\
& & =C^{r,s}_{n,k}\int_{\eta\cap\Nor K}x^ru^s\sum_{i=1}^{n-1}b_{i}(x,u)^2
\sum_{|I|=n-1-k\atop i\notin I}\frac{\prod_{j\in I} k_{j}(x,u)}{\mathbb{K}(x,u)}\,\mathcal{H}^{n-1}(\D(x,u)).\nonumber
\end{eqnarray}

If $k=1$, then 
$$
\phi^{r,s,1}_1(K,\eta)=C^{r,s}_{n,1}\int_{\eta\cap\Nor K} x^ru^s \sum_{i=1}^{n-1}b_{i}(x,u)^2 \frac{\prod_{j:j\neq i} k_{j}(x,u)}{\mathbb{K}(x,u)}\,\mathcal{H}^{n-1}(\D(x,u)),
$$
and for $k=n-2$, we have
$$
\phi^{r,s,1}_{n-2}(K,\eta)=C^{r,s}_{n,n-2}\int_{\eta\cap\Nor K}x^ru^s\sum_{i=1}^{n-1}b_{i}(x,u)^2\sum_{j:j\neq i}
\frac{k_{j}(x,u)} {\mathbb{K}(x,u)}\,\mathcal{H}^{n-1}(\D(x,u)).
$$

For $n=3$, these two special cases coincide and we get
$$
\phi^{r,s,1}_{1}(K,\eta)=C^{r,s}_{3,1}\int_{\eta\cap\Nor K} x^ru^s \frac{k_1(x,u)b_{2}(x,u)^2+k_2(x,u)b_{1}(x,u)^2}
{\mathbb{K}(x,u)}\,\mathcal{H}^{2}(\D(x,u)).
$$

For a convex body of class $C^2$, we write $u_x$ for the unique exterior unit normal of $K$ at the boundary point $x\in\partial K$ of $K$. An application of the coarea 
formula (together with Lemma 3.1 from \cite{Hug98b}) then yields
$$
\phi^{r,s,1}_k(K,\eta)=C^{r,s}_{n,k}\int_{\partial K}\mathbf{1}_\eta(x,u_x)x^ru_x^s\sum_{i=1}^{n-1}b_{i}(x)^2
\sum_{|I|=n-1-k\atop i\notin I}\prod_{j\in I} k_{j}(x)\,\mathcal{H}^{n-1}(\D x),
$$
where the $k_j(x)$ are the principal curvatures and the unit vectors $b_j(x)$ give the principal directions of curvature of $K$ at $x\in\partial K$. In particular, for a convex body $K$ in $\R^3$ with a $C^2$ boundary we get 
$$
\phi^{r,s,1}_{1}(K,\eta)=C^{r,s}_{3,1}\int_{\partial K }\mathbf{1}_\eta(x,u_x)x^ru_x^s 
\left(k_1(x)b_{2}(x)^2+k_2(x)b_{1}(x)^2  \right)\,\mathcal{H}^{2}(\D x).
$$

\section{Proof of Theorem \ref{Theorem2.3}}\label{sec5}

For the proof of Theorem \ref{Theorem2.3}, it suffices to prove the following.

\begin{theorem}\label{Theorem5.1}
Let $p\in\N_0$. Let $\Gamma:\Kn\times\B(\Sigma^n)\to\T^p$ be a mapping with the following properties.\\
$\rm (a)$ $\Gamma(K,\cdot)$ is a $\T^p$-valued measure, for each $K\in\Kn$;\\ 
$\rm (b)$ $\Gamma$ is translation invariant and rotation covariant;\\
$\rm (c)$ $\Gamma$ is locally defined;\\
$\rm (d)$ $\Gamma$ is weakly continuous.\\
Then $\Gamma$ is a linear combination, with constant coefficients, of the mappings $Q^m\phi^{0,s,j}_k$, where $m,s\in\N_0$ and $j\in\{0,1\}$ satisfy $2m+2j+s=p$ and where $k\in\{0,\dots,n-1\}$, but $j=0$ if $k\in\{0,n-1\}$.
\end{theorem}

Since we know that the mappings $Q^m\phi_k^{p,s,j}$ with $j\in\{0,1\}$ have the properties (a), (b), (c), (d) and since linear independence follows from Theorem \ref{Theorem3.1}, Theorem \ref{Theorem2.3} can be derived from Theorem \ref{Theorem5.1} in the same way as, in Section \ref{sec3}, Theorem \ref{Theorem2.2} was derived from Theorem \ref{Theorem3.2}. For this, we use the weak continuity to extend (\ref{n3.1}), for $j\in\{0,1\}$, from polytopes to general convex bodies.

For the proof of Theorem \ref{Theorem5.1} we note first that, by Lemma \ref{Lemma3.5}, it is sufficient to prove the theorem under the additional assumption that $\Gamma$ is homogeneous of some fixed degree $k\in\{0,\dots,n-1\}$. We assume this and then distinguish two cases.

If $k\in\{0,n-1\}$, then Theorem \ref{Theorem3.2} shows that on polytopes $P$ the mapping $\Gamma$ is of the form
$$ \Gamma(P,\cdot) = \sum_{m,s\ge 0 \atop 2m+s=p} c_{ms}  Q^m\phi^{0,s}_k(P,\cdot) $$
with constants $c_{ms}$. From the weak continuity of $\Gamma$ and of $\phi^{0,s}_k$ it follows that 
$$ \Gamma(K,\cdot) = \sum_{m,s\ge 0 \atop 2m+s=p} c_{ms}  Q^m\phi^{0,s}_k(K,\cdot) $$
for all $K\in\Kn$.

Now let $k\in\{1,\dots,n-2\}$. Thus, from now on we are dealing only with dimensions $n\ge 3$. By Theorem \ref{Theorem3.2}, on polytopes $P$ the mapping $\Gamma$ is of the form
$$ \Gamma(P,\cdot) = \sum_{m,j,s\ge 0 \atop 2m+2j+s=p} c_{mjs}  Q^m\phi^{0,s,j}_k(P,\cdot) $$
with constants $c_{mjs}$. Since $\Gamma$ and the mappings $\phi_k^{0,s,0}$ and $\phi_k^{0,s,1}$ are weakly continuous, the mapping $\Gamma'$ defined by 
$$ \Gamma':=\Gamma-\sum_{m,j,s\ge 0,\,j\le 1 \atop 2m+2j+s=p} c_{mjs}  Q^m\phi^{0,s,j}_k$$
has again the properties (a) -- (d) of Theorem \ref{Theorem5.1}, and on polytopes $P$ it is of the form
$$ \Gamma'(P,\cdot) = \sum_{m,s\ge 0,\,j\ge 2 \atop 2m+2j+s=p} c_{mjs}  Q^m\phi^{0,s,j}_k(P,\cdot). $$
We have to show that here all the remaining constants $c_{mjs}$ are zero. 

The main idea of the proof is to construct a particular sequence of polytopes which converges to a convex body that has more rotational symmetries than the approximating polytopes. The mapping $\Gamma'$ must be covariant under the rotations mapping the limit body into itself. If $\Gamma'$ is not identically zero, it can be shown that for our special choice of the approximating polytopes this covariance is violated so strongly that it is still violated in the weak limit, which is a contradiction. This approach requires some preparations. In the following, we write again $\Gamma$ instead of $\Gamma'$.

Let $(e_1,\dots,e_n)$ be an orthonormal basis of ${\mathbb R}^n$ and let ${\mathbb R}^{n-1}$ be the subspace spanned by $e_1,\dots,e_{n-1}$. 

\vspace{2mm}

\noindent{\bf Definition.} Let $\omega\in{\mathcal B}({\mathbb S}^{n-1})$ and $0<\varepsilon<1$. We say that $\omega$ is $\varepsilon$-{\em close to} $-e_n$ if each $u\in\omega$ satisfies
\begin{equation}\label{34a}
\langle u,-e_n\rangle > 1-\varepsilon
\end{equation}
and
\begin{equation}\label{35}
|\langle u,a\rangle| < \varepsilon\|a\| \quad\mbox{for each }a\in{\mathbb R}^{n-1}.
\end{equation}

\vspace{2mm}

For example, if $\mu>0$ is sufficiently small (depending on $\varepsilon$), then the nonempty, open set
$$ \omega = \{u\in \Sn:\langle u,-e_n\rangle >1-\mu\}$$
is $\varepsilon$-close to $-e_n$.

In the following, we write
$$ \Gamma(K,f):= \int_{\Sigma^n} f(u)\, \Gamma(K,\D (x,u))$$
for $K\in \Kn$ and any continuous real function $f$ on the unit sphere ${\mathbb S}^{n-1}$. For a polytope $P\in \Pn$, we define
$$ W_k(P,f):= \sum_{F\in{\mathcal F}_k(P)} \Ha^k(F)\int_{\nu(P,F)}f\,\D\Ha^{n-k-1}.$$
Using the $k$th area measure $\Psi_k(P,\cdot)=\Lambda_k(P,\R^n\times\cdot)$ (see \cite[(4.20), (4.24)]{Sch14}), this can also be written as
\begin{equation}\label{35a}
W_k(P,f) = \omega_{n-k}\int_{\Sn} f\,\D\Psi_k(P,\cdot).
\end{equation}

For polytopes $P$ and for suitable $f$ and $E$, the following lemma approximates $\Gamma(P,f)(E)$ by a simpler expression.

\vspace{2mm}

\begin{lemma}\label{Lemma5.1}
Let $k\in\{1,\dots,n-2\}$. Suppose that $\Gamma$ satisfies the assumptions listed in Theorem $\ref{Theorem5.1}$ and that for $P\in{\mathcal P}^n$ it is of the form
\begin{equation}\label{31a} 
\Gamma(P,\cdot) = \sum_{m,s\ge 0,\,j\ge 2 \atop 2m+2j+s=p} c_{mjs}  Q^m\phi^{0,s,j}_k(P,\cdot)
\end{equation} 
with constants $c_{mjs}$ which are not all zero. Let $s_0$ be the smallest number $s$ for which $c_{mjs}\not=0$ for some $m,j$, set $q:=(p-s_0)/2$ and $c_j:=((2q)!s_0!/p!)c_{(q-j)js_0}C_{n,k}^{0,s_0}$ $($recall definition $(\ref{2.7a}))$. Let $d$ be the largest $j\in \{2,\dots,q\}$ for which $c_j\not=0$. 

Let $\omega\in{\mathcal B}({\mathbb S}^{n-1})$ and $0<\varepsilon<1$ be such that $\omega$ is $\varepsilon$-close to $-e_n$. Let $f$ be a nonnegative, continuous real function on ${\mathbb S}^{n-1}$ with support in $\omega$. For $P\in{\mathcal P}^n$, define $\Delta(P,f)\in{\mathbb T}^{2q}$ by
\begin{equation}\label{5.2ab}  
\Delta(P,f) := \sum_{j=2}^d c_j Q^{q-j}\sum_{F\in{\mathcal F}_k(P)} Q^j_{L(F)}\Ha^k(F) \int_{\nu(P,F)} f\,\D\Ha^{n-k-1}.
\end{equation}

Let
$$ E':=(\underbrace{a,\dots,a}_{2q})\quad\mbox{with } a\in{\mathbb R}^{n-1}, \|a\|=1,$$
and
\begin{equation}\label{5.2a} 
E:=(b_1,\dots,b_p):=(\underbrace{a,\dots,a}_{2q},\underbrace{-e_n,\dots,-e_n}_{s_0}).
\end{equation}
Then
\begin{equation}\label{5.2c}  
\left|\Gamma(P,f)(E)-\Delta(P,f)(E')\right|\le C_3 W_k(P,f)\varepsilon
\end{equation}
with a constant $C_3$ depending only on $\Gamma$.
\end{lemma}

\begin{proof} For a polytope $P$ and any set $\eta={\mathbb R}^n\times\omega'$ with $\omega'\in{\mathcal B}({\mathbb S}^{n-1})$, the representation (\ref{31a}) can explicitly be written as
\begin{equation} 
\Gamma(P,\eta)= \sum_{m,s\ge 0,\,j\ge 2 \atop 2m+2j+s=p} c_{mjs}   C_{n,k}^{0,s} \sum_{F\in{\mathcal F}_k(P)} \Ha^k(F)\int_{\omega'\cap \nu(P,F)} Q^m Q_{L(F)}^j u^s\,{\mathcal H}^{n-k-1}(\D u),
\end{equation}
according to (\ref{2.6}). Therefore,
\begin{equation}\label{32}  
\Gamma(P,f)= \sum_{m,s\ge 0,\,j\ge 2 \atop 2m+2j+s=p} c_{mjs}   C_{n,k}^{0,s} \sum_{F\in{\mathcal F}_k(P)} \Ha^k(F)\int_{\nu(P,F)} Q^m Q_{L(F)}^j u^sf(u)\,{\mathcal H}^{n-k-1}(\D u).
\end{equation}

In the following estimates, we need only consider vectors $u\in\omega$, since $f$ has its support in $\omega$. Let $u\in\omega$. For a $p$-tuple $E$ as given by (\ref{5.2a}), we have
\begin{eqnarray*} 
& & (Q^mQ_{L(F)}^ju^s)(E)\\
& &= \frac{1}{p!} \sum_{\sigma\in S(p)} Q^m(b_{\sigma(1)},\dots,b_{\sigma(2m)})Q_{L(F)}^j (b_{\sigma(2m+1)},\dots,b_{\sigma(2m+2j)}) u^s(b_{\sigma(2m+2j+1)},\dots,b_{\sigma(p)}),
\end{eqnarray*}
where $S(p)$ is the group of permutations of $1,\dots,p$. If at least one of the arguments of $u^s$, the
vectors $b_{\sigma(2m+2j+1)},\dots,b_{\sigma(p)}$, is not equal to $- e_n$ and hence is equal to $a$, then
\begin{equation}\label{33} 
|u^s(b_{\sigma(2m+2j+1)},\dots,b_{\sigma(p)})|=|\langle u,b_{\sigma(2m+2j+1)}\rangle| \cdots |\langle u,b_{\sigma(p)}\rangle|\le \varepsilon
\end{equation}
by (\ref{35}), since we have assumed that $\omega$ is $\varepsilon$-close to $-e_n$ and $\|a\|=1$. When $s>s_0$, this holds for all $\sigma\in S(p)$. The absolute value of $Q^m(\cdot)Q_{L(F)}^j(\cdot)$ in the last sum above is at most $1$. Hence, we have
$$ |(Q^mQ_{L(F)}^ju^s)(E)| \le \varepsilon \qquad\mbox{for }s>s_0.$$ 
For each fixed $s>s_0$, this yields the estimate
\begin{eqnarray*}
& & \left| \sum_{m\ge 0,\,j\ge 2 \atop 2m+2j=p-s} c_{mjs}C_{n,k}^{0,s} \sum_{F\in{\mathcal F}_k(P)} \Ha^k(F)\int_{\nu(P,F)} (Q^m Q_{L(F)}^j u^{s})(E)f(u)\,{\mathcal H}^{n-k-1}(\D u)\right|\\
& & \le \sum_{m\ge 0,\,j\ge 2 \atop 2m+2j=p-s} |c_{mjs}| C_{n,k}^{0,s} \sum_{F\in{\mathcal F}_k(P)} \Ha^k(F)\int_{\nu(P,F)} \varepsilon f\,\D\Ha^{n-k-1}\\
& & \le C_1W_k(P,f)\varepsilon
\end{eqnarray*}
with a constant $C_1$ that (for given dimension) depends only on $\Gamma$. For $s=s_0$ we have
$$ |u^{s_0}(b_{\sigma(2m+2j+1)},\dots,b_{\sigma(p)})|\le \varepsilon $$
if at least one of the vectors $b_{\sigma(2m+2j+1)},\dots,b_{\sigma(p)}$ is not equal to $- e_n$, and otherwise
$$ u^{s_0}(b_{\sigma(2m+2j+1)},\dots,b_{\sigma(p)})=\langle u,-e_n\rangle^{s_0}.$$
Hence, we obtain
$$ (Q^m Q_{L(F)}^j u^{s_0})(E) = \frac{(2q)!s_0!}{p!}(Q^mQ_{L(F)}^j)(E')\langle u,-e_n\rangle^{s_0} + R_1$$
with
$$ |R_1|\le\left[1-\binom{p}{s_o}^{-1}\right]\varepsilon.$$
For the vectors $u\in\omega$ we have the estimate $ 1-\varepsilon \le \langle u,-e_n\rangle \le 1$ by (\ref{34a}) and hence we can write 
$$ (Q^m Q_{L(F)}^j u^{s_0})(E) = \frac{(2q)!s_0!}{p!}(Q^mQ_{L(F)}^j)(E') + R_2$$
with
$$ |R_2|\le C_2\varepsilon,$$
where the constant $C_2$ depends only on $\Gamma$.

Splitting the sum in (\ref{32}) in the form $\sum_{m,j,s} = \sum_{m,j,s_0}+\sum_{m,j, s>s_0}$, we conclude that
\begin{eqnarray*}
& & \Gamma(P,f)(E)\\
& & = \frac{(2q)!s_0!}{p!} \sum_{m\ge 0,\,j\ge 2\atop 2m+2j=p-s_0} c_{mjs_0}C_{n,k}^{0,s_0} \sum_{F\in{\mathcal F}_k(P)}(Q^mQ_{L(F)}^j)(E') \Ha^k(F) \int_{\nu(P,F)}f\,\D{\mathcal H}^{n-k-1} + R_3
\end{eqnarray*}
with 
$$ |R_3|\le C_3W_k(P,f)\varepsilon,$$
where $C_3$ depends only on $\Gamma$. With $((2q)!s_0!/p!)c_{(q-j)js_0}C_{n,k}^{0,s_0}=c_j$, we obtain
\begin{eqnarray*} 
\Gamma(P,f)(E) &=& \sum_{j=2}^q c_j\sum_{F\in{\mathcal F}_k(P)} (Q^{q-j}Q_{L(F)}^j)(E')
\Ha^k(F)\int_{\nu(P,F)}f \,\D{\mathcal H}^{n-k-1} + R_3\\ 
&=& \Delta(P,f)(E')+R_3
\end{eqnarray*}
(recall that $c_j\neq 0$ for some $j\in \{2,\ldots,q\}$ and that $d$ was defined as the largest number $j\in \{2,\dots,q\}$ such that $c_j\not=0$).
\end{proof}

As already indicated, we construct a sequence of polytopes that converges to a convex body having more rotational symmetries than the approximating polytopes. If $\Gamma\not\equiv 0$, then on these polytopes the mapping $\Gamma$ violates the rotation covariance in a way that is preserved under the weak limit. This contradiction will show that $\Gamma\equiv 0$. We turn to the construction of the polytopes.

Recall that $(e_1,\dots,e_n)$ is the standard orthonormal basis of ${\mathbb R}^n$ and that the linear hull ${\rm lin}\{e_1,\dots,e_{n-1}\}$ is identified with ${\mathbb R}^{n-1}$. We write the points $y\in {\mathbb R}^n$ in the form
$$ y= y_1e_1+\dots +y_ne_n= (y_1,\dots, y_n),$$
thus $(x_1,\dots,x_{n-1},0)\in{\mathbb R}^{n-1}$. 

We define the lifting map $\ell:{\mathbb R}^{n-1}\to{\mathbb R}^n$ by
\begin{equation}\label{lift} 
\ell(x):= x+\|x\|^2e_n \qquad\mbox{for } x\in{\mathbb R}^{n-1}.
\end{equation}
Then ${\mathcal R}:= \ell({\mathbb R}^{n-1})$ is a paraboloid of revolution. %For later application we note that for 

Let $t>0$ be given and consider the lattice $2t{\mathbb Z}^{n-1}$ of all points 
$$2t(m_1,\dots,m_{n-1},0)\in{\mathbb R}^{n-1}$$ 
with integers $m_1,\dots,m_{n-1}$. The points of $2t{\mathbb Z}^{n-1}$ are the vertices of a tessellation of ${\mathbb R}^{n-1}$ into $(n-1)$-cubes, which together with all their faces form a polytopal complex, which we denote by ${\mathcal C}_t$. 

We consider the polyhedral set
$$ R_t:= {\rm conv}\,\ell(2t{\mathbb Z}^{n-1}).$$
For given $z\in\R^{n-1}$, we define an affine map $\alpha_z:\R^{n-1}\to\R^n$ by
\begin{equation}\label{aff} 
\alpha_z(y):= y+2\langle z,y\rangle e_n +\left(r_t^2-\|z\|^2\right)e_n, \qquad y\in\R^{n-1},
\end{equation}
where $r_t$ ($=t\sqrt{n-1}$) is the radius of the sphere through the vertices of a cube in ${\mathcal C}_t$. Then
$$ \alpha_z(z+x)=\ell(z+x)+\left(r_t^2-\|x\|^2\right)e_n$$
for all $x\in\R^{n-1}$. Hence, if $\|x\|=r_t$, then
$$ \alpha_z(z+x)=\ell(z+x).$$
Let $C_z$ be an $(n-1)$-cube of the complex ${\mathcal C}_t$, with center $z$. On the vertices of $C_z$, the mapping $\ell$ coincides with the affine map $\alpha_z$. Therefore, the $\ell$-images of the vertices of $C_z$ lie in the hyperplane $H:=\alpha_z({\mathbb R}^{n-1})$. Every point of the lattice $2t{\mathbb Z}^{n-1}$ which is not a vertex of $C_z$ is mapped by $\ell$ into the `upper' open halfspace bounded by $H$. It follows that $H$ is a supporting hyperplane of $R_t$. Therefore, the convex hull of the $\ell$-images of the vertices of $C_z$ is a facet of $R_t$. (A more general version of this observation is well known in the theory of Voronoi and Delaunay tessellations; see \cite{For04}, \cite[Sec.~4.4]{DO11}, \cite[Sec.~7]{JT13}, for example.) From this it follows further that each face $G$ of ${\mathcal C}_t$ is in one-to-one correspondence with a face $F$ of $R_t$ of the same dimension, in such a way that $F$ is the convex hull of the $\ell$-images of the vertices of $G$. We say in this case that $F$ {\em lies above} $G$. In particular, $R_t$ has no other faces than those lying above the faces of ${\mathcal C}_t$. For a given face $F$ of $R_t$, we denote by $F^\Box$ the face of ${\mathcal C}_t$ above which it lies. If $F$ is the facet of $R_t$ for which $F^\Box=C_z$, then $F=\alpha_z(C_z)$.

For $h>0$, let
$$ H^-_h := \{y\in{\mathbb R}^n:\langle y,e_n\rangle \le h\}.$$
We define a convex body of revolution by
$$ K_h:= \left\{y\in{\mathbb R}^n:y_n  \ge y_1^2+\dots+y_{n-1}^2\right\}\cap H^-_h.$$
This is the intersection of the epigraph of $\ell$ with the closed halfspace $H^-_h$. Further, we define a polytope $P_{h,t}$ by
$$ P_{h,t}:= R_t\cap H^-_h.$$
Then $P_{h,t}\to K_h$  for $t\to 0$. 

By $\omega_h$ we denote the spherical image (that is, the set of outer unit normal vectors) of the part of $K_h$ that lies in the interior of the halfspace $H^-_{h/2}$.

Now let $0<\varepsilon<1$ be given. We choose $h$ in dependence of $\varepsilon$, in the following way.

Recall that $\pi_L:{\mathbb R}^n\to L$ denotes the orthogonal projection to a subspace $L$. Let $\lambda_z:{\mathbb R}^{n-1}\to{\mathbb R}^n$ be the linear part of the affine map $\alpha_z$ used above, that is, $$\lambda_z(x):= x+2\langle z,x \rangle e_n\qquad\mbox{for } x\in\R^{n-1}.$$ 
Since $\lambda_z$ for $z=0$ is just the inclusion map ${\mathbb R}^{n-1} \hookrightarrow {\mathbb R}^n$, we can choose $h>0$ so small that for $z\in\R^{n-1}$ with $\|z\|^2\le h$ the following holds. For every linear subspace $L$ of ${\mathbb R}^{n-1}$ and for any unit vector $a\in{\mathbb R}^n$, we have
\begin{equation}\label{13} 
\left|\|\pi_L a\|^2 -\|\pi_{\lambda_z L} a\|^2\right| \le \varepsilon.
\end{equation}
Moreover, we choose $h$ so small that the set $\omega_h$ is $\varepsilon$-close to $-e_n$.

We point out that $K_h,P_{h,t},\omega_h$ all depend on $\varepsilon$, although this is not made explicit by the notation.

We define
$$ {\mathcal F}_{k,h}(P_{h,t}):=\{F\in{\mathcal F}_k(P_{h,t}): \omega_h\cap\nu(P_{h,t},F)\not=\emptyset\}.$$

For $u\in\omega_h$, let $H(K,u)$ denote the supporting hyperplane of the convex body $K$ with outer normal vector $u$. Let $H_h$ be the boundary hyperplane of the halfspace $H^-_h$. There is a number $\delta>0$ such that each hyperplane $H(K_h,u)$ with $u\in\omega_h$ has distance at least $\delta$ from the set $K_h\cap H_h$. Therefore, there is a number $t_0>0$ (depending on $\varepsilon$) such that for $0<t<t_0$ each supporting hyperplane $H(P_{h,t},u)$ with $u\in\omega_h$ has distance at least $\delta/2$ from $K_h\cap H_h$. For these $t$, a face $F\in{\mathcal F}_{k,h}(P_{h,t})$ cannot contain a point of $H_h$ and hence must be a face of $R_t$ lying above some face $F^\Box$ of ${\mathcal C}_t$. Decreasing $t_0$, if necessary, we can further assume that each such face $F^\Box$ is a face of some cube $C_z$ of ${\mathcal C}_t$ whose center satisfies $\|z\|^2\le h$. We assume in the following that $0<t<t_0$.

Let ${\mathcal L}_k$ denote the set of all linear subspaces spanned by any $k$ vectors of $e_1,\dots,e_{n-1}$. This set of $k$-dimensional coordinate subspaces of $\R^{n-1}$ contains $b(n,k)=\binom{n-1}{k}$ elements. We numerate them by $L_1,\dots,L_{b(n,k)}$. By ${\mathcal F}^i_{k,h}(P_{h,t})$ we denote the set of $k$-faces $F\in{\mathcal F}_{k,h}(P_{h,t})$ for which $F^\Box$ is parallel to $L_i$. 

Now let $f$ be a nonnegative, continuous function on ${\mathbb S}^{n-1}$ that is invariant under all rotations fixing $e_n$, has its support in $\omega_h$ and is not identically zero.  

The number
\begin{equation}\label{36a} 
W^i_k(P_{h,t},f):= \sum_{F\in{\mathcal F}^i_{k,h}(P_{h,t})} \Ha^k(F)\int_{\nu(P_{h,t},F)}f\, \D{\mathcal H}^{n-k-1}
\end{equation}
is independent of $i$, since the polytope $P_{h,t}$ and the function $f$ are invariant under the rotations permuting the basis vectors $e_1,\dots,e_{n-1}$ and fixing $e_n$. Since
$$ \sum_{i=1}^{b(n,k)}\sum_{F\in{\mathcal F}^i_{k,h}(P_{h,t})} \Ha^k(F)\int_{\nu(P_{h,t},F)}f\, \D{\mathcal H}^{n-k-1} =
\sum_{F\in{\mathcal F}_k(P_{h,t})} \Ha^k(F)\int_{\nu(P_{h,t},F)}f\, \D{\mathcal H}^{n-k-1},$$
where we used that $f$ has its support in $\omega_h$, we have
\begin{equation}\label{36}
b(n,k) W^i_k(P_{h,t},f)= W_k(P_{h,t},f).
\end{equation}
This finishes the construction of the polytopes $P_{h,t}$ and the description of their properties.

Suppose now that $\Gamma$ is a mapping with the properties listed in Theorem \ref{Theorem5.1} and such that for $P\in{\mathcal P}^n$ it is given by
$$ \Gamma(P,\cdot) = \sum_{m,s\ge 0,\,j\ge 2 \atop 2m+2j+s=p} c_{mjs}  Q^m\phi^{0,s,j}_k(P,\cdot) $$
with constants $c_{mjs}$ which are not all zero. Let $f$ be a function as described above. 

As in Lemma \ref{Lemma5.1}, we define $s_0$ as the smallest number $s$ for which $c_{mjs}\not=0$ for some $m,j$. Definition (\ref{5.2ab}) for the polytope $P_{h,t}$ reads
\begin{equation}\label{50a}  
\Delta(P_{h,t},f) = \sum_{j=2}^d c_j Q^{q-j}\sum_{F\in{\mathcal F}_{k,h}(P_{h,t})} Q^j_{L(F)}\Ha^k(F) \int_{\nu(P,F)} f\,\D\Ha^{n-k-1},
\end{equation}
where only faces $F\in{\mathcal F}_{k,h}(P_{h,t})$ appear, since $f$ has its support in $\omega_h$. Let
\begin{equation}\label{40b}
E:= (E',\underbrace{-e_n,\dots,-e_n}_{s_0}),
\end{equation}
where
$$ E':=(\underbrace{a,\dots,a}_{2q})\quad\mbox{with } a\in{\mathbb R}^{n-1},\, \|a\|=1.$$
Since $\omega_h$ is $\varepsilon$-close to $-e_n$, it follows from Lemma \ref{Lemma5.1} that
\begin{equation}\label{51}
\left|\Gamma(P_{h,t},f)(E)-\Delta(P_{h,t},f)(E')\right|\le C_3W_k(P_{h,t},f)\varepsilon.
\end{equation}

Let $F\in{\mathcal F}_{k,h}(P_{h,t})$. The face $F$ lies above some $k$-face $F^\Box$ of ${\mathcal C}_t$. The face $F^\Box$ belongs to some cube $C_z$ with center $z$ satisfying $\|z\|^2\le h$, and $F$ is a translate of $\lambda_zF^\Box$, therefore $L(F)=\lambda_z(L(F^\Box))$. 

Let $a\in{\mathbb R}^{n-1}$ be a unit vector. By (\ref{13}),
$$ \left|\|\pi_L a\|^2-\|\pi_{\lambda_z L} a\|^2\right|\le\varepsilon.$$
This yields
$$ (Q^{q-j}Q_{L(F)}^j)(E') = (Q^{q-j}Q_{L(F^\Box)}^j)(E')+R_4$$
with $$|R_4|\le C_4\varepsilon,$$
where the constant $C_4$ can be chosen to depend only on $\Gamma$. Together with (\ref{50a}) and (\ref{51}) this yields
\begin{eqnarray}\label{41} 
& & \Gamma(P_{h,t},f)(E)\\
& & =  \sum_{j=2}^d c_j\sum_{F\in{\mathcal F}_{k,h}(P_{h,t})} \left(Q^{q-j}Q_{L(F^\Box)}^j\right)(E')
\Ha^k(F)\int_{\nu(P_{h,t},F)} f\,\D{\mathcal H}^{n-k-1} +R_5(E) \nonumber
\end{eqnarray}
with 
$$ |R_5(E)|\le C_5W_k(P_{h,t},f)\varepsilon,$$
where $C_5$ depends only on $\Gamma$. (We have written $R_5(E)$, for fixed $P_{h,t},f$, since later we shall have to distinguish between remainder terms for different arguments $E$.)

For each $F\in{\mathcal F}_{k,h}(P_{h,t})$ we have $L(F^\Box)=L_i$ for suitable $i\in \{1,\dots,b(n,k)\}$, hence, using (\ref{36}),
\begin{eqnarray}\label{53} 
& & \Gamma(P_{h,t},f)(E)\\
& & = \sum_{j=2}^d c_j  \sum_{i=1}^{b(n,k)}\sum_{F\in{\mathcal F}^i_{k,h}(P_{h,t})} (Q^{q-j}Q_{L_i}^j)(E')
\Ha^k(F)\int_{\nu(P_{h,t},F)} f\,\D{\mathcal H}^{n-k-1} +R_5(E) \nonumber\\
& &= b(n,k)^{-1}W_k(P_{h,t},f)\left(\sum_{j=2}^d c_j Q^{q-j}\sum_{i=1}^{b(n,k)} Q_{L_i}^j\right)(E')+R_5(E).\nonumber
\end{eqnarray}

Before we continue, it seems appropriate to point out why we cannot finish the proof shortly. For a shorter proof, we would need the following. Suppose that  $d\le q$. If the polynomial 
$$ \sum_{j=2}^d c_j(x_1^2+\cdots+x_n^2)^{q-j} \sum_{I\subset \{1,\dots,n\},|I|=k}\left(\sum_{i\in I} x_i^2\right)^j$$
is invariant under $SO(n)$, then $c_2=\cdots=c_d=0$. However, this is not true. The simplest example is obtained for $n=2$, $k=1$, $q=d=3$, by
$$ P(x_1,x_2)=c_2(x_1^2+x_2^2)\left[(x_1^2)^2+(x_2^2)^2\right]+c_3\left[(x_1^2)^3+(x_2^2)^3\right].$$
With $x_1^2+x_2^2=1$ we get
$$ P(x_1,x_2)=(2c_2+3c_3)x_1^4-(2c_2+3c_3)x_1^2+(c_2+c_3).$$
Hence, if $2c_2+3c_3=0$, then, for arbitrary $x_1,x_2$,
$$ P(x_1,x_2)=(c_2+c_3)(x_1^2+x_2^2)^3.$$ 
Thus, $P$ is rotation invariant, but not necessarily $c_2=c_3=0$. There are many other examples of this kind. The existence of these examples forces us to perform a case distinction, involving some geometric arguments.

Our approach requires, as it turns out, to consider dimensions three and higher separately. First we consider the case $n\ge 4$ and prove the following lemma.

\vspace{2mm}

\begin{lemma}\label{Lemma5.2} Let $n\ge 4$. Under the assumptions made above, there exist a convex body $K\in{\mathcal K}^n$, a continuous function $f$ on $\Sn$, a $p$-tuple $E$, and a rotation $\vartheta\in{\rm O}(n)$ such that $K$ and $f$ are invariant under $\vartheta$, but $\Gamma(K,f)(\vartheta E)\not=\Gamma(K,f)(E)$.
\end{lemma}

If this is proved, then the invariance of $K$ and $f$ under $\vartheta$ and the rotation covariance of $\Gamma$ give
$$ \Gamma(K,f)(\vartheta E)=\Gamma(\vartheta K,\vartheta f)(\vartheta E) = \Gamma(K,f)(E).$$
This is a contradiction, which finishes the proof of Theorem 5.1 for $n\ge 4$. 

\vspace{2mm}

\noindent{\em Proof of Lemma} 5.2. Let ${\rm O}(n,e_n)$ denote the group of rotations of $\R^n$ that fix $e_n$. First we want to show that the tensor
$$ \Upsilon:= \sum_{j=2}^d c_j Q^{q-j}\sum_{i=1}^{b(n,k)} Q_{L_i}^j $$
is not invariant under ${\rm O}(n,e_n)$. For $x=(x_1,\dots, x_{n-1},0)\in\R^n$ with $\|x\|=1$ we have
$$ p_\Upsilon(x):= \Upsilon(\underbrace{x,\dots,x}_{2q}) =  \sum_{j=2}^d c_j  \sum_{I\subset\{1,\dots,n-1\},\,|I|=k}\left(\sum_{i\in I} x_i^2\right)^j.$$

First let $d$ be even. Consider $x= (x_1,x_2,0,\dots,0)\in\R^n$ with $\|x\|=1$. For such $x$,
$$ p_\Upsilon(x)=\sum_{j=2}^d c_j\left\{\binom{n-3}{k-2}(x_1^2+x_2^2)^j+\binom{n-3}{k-1}(x_1^{2j}+x_2^{2j})\right\}.$$
Here and below we make use of the convention that $\binom{n}{m}=0$ for $m<0$. With $x(\lambda)= (\lambda,\sqrt{1-\lambda^2},0,\dots,0)$, $\lambda\in[0,1]$, we get
$$ p_\Upsilon(x(\lambda)) = \sum_{j=2}^d c_j\left\{\binom{n-3}{k-2}+\binom{n-3}{k-1}(\lambda^{2j}+(1-\lambda^2)^j)\right\}=2c_d\binom{n-3}{k-1}\lambda^{2d}+\dots,$$
where the dots stand for terms where the exponent of $\lambda$ is less than $2d$. Since $\lambda\mapsto  p_\Upsilon(x(\lambda))$ is not a constant function, the polynomial $p_\Upsilon$ is not constant on unit vectors in $\R^{n-1}$ and hence is not invariant under ${\rm O}(n,e_n)$.

Now let $d$ be odd. For $x=(x_1,x_2,x_3,0,\dots,0)\in\R^n$ with $\|x\|=1$ we have
\begin{eqnarray*}
p_\Upsilon(x) &=& \sum_{j=2}^d c_j \Bigg\{\binom{n-4}{k-3}
 +\binom{n-4}{k-2} \left[\left(x_1^2+x_2^2\right)^j+ \left(x_1^2+x_3^2\right)^j +\left(x_2^2+x_3^2\right)^j \right]\\
&&+ \binom{n-4}{k-1}\left(x_1^{2j}+x_2^{2j}+x_3^{2j}\right)\Bigg\}.
\end{eqnarray*}
With
$$ x(\lambda,\mu)= \left(\lambda,\mu\sqrt{1-\lambda^2},\sqrt{1-\mu^2}\sqrt{1-\lambda^2},0,\dots,0\right)\in\R^n,\qquad \lambda,\mu\in[0,1],$$
we get by an elementary calculation that
\begin{eqnarray*}
p_\Upsilon(x(\lambda,\mu)) &=& dc_d\left[\binom{n-4}{k-2}-\binom{n-4}{k-1}\right]\mu^{2d-2}\lambda^{2d}+\dots,
\end{eqnarray*}
where the dots stand for terms in which the exponent of $\lambda$ is less than $2d$ or the exponent of $\mu$ is less than $2d-2$. If $2k\not=n-1$, then $\binom{n-4}{k-2}-\binom{n-4}{k-1}\not=0$, hence the function
\begin{equation}\label{43}
(\lambda,\mu)\mapsto p_\Upsilon(x(\lambda,\mu)),\quad \lambda,\mu\in[0,1],
\end{equation}
is not constant. Let us first assume that $2k\not=n-1$.

In both cases, $d$ even and $d$ odd, we have seen that the polynomial $p_\Upsilon$ and hence the tensor $\Upsilon$ is not ${\rm O}(n,e_n)$ invariant. Hence, there are a unit vector $a\in\R^{n-1}$ and a rotation $\vartheta\in{\rm O}(n,e_n)$ such that the $(2q)$-tuple $E':= (a,\dots,a)$ satisfies
$$ |\Upsilon(\vartheta E')-\Upsilon(E')|=M>0.$$ 
By (\ref{53}), for $E:= (E',-e_n,\dots,-e_n)$ we have 
$$ \Gamma(P_{h,t},f)(E) = b(n,k)^{-1}W_k(P_{h,t},f)\Upsilon(E')+R_5(E)$$
with $|R_5(E)|\le C_5 W_k(P_{h,t},f)\varepsilon$, and similarly for $\vartheta E$. Thus, we get
\begin{eqnarray*}
& & |\Gamma(P_{h,t},f)(\vartheta E)-\Gamma(P_{h,t},f)(E)|\\
& & = |b(n,k)^{-1} W_k(P_{h,t},f)\Upsilon(\vartheta E')+ R_5(\vartheta E)-b(n,k)^{-1}W_k(P_{h,t},f)\Upsilon(E')- R_5(E)|\\
& & \ge b(n,k)^{-1} W_k(P_{h,t},f)|\Upsilon(\vartheta E')-\Upsilon(E')|-|R_5(\vartheta E)|-|R_5(E)|\\
& & \ge b(n,k)^{-1} W_k(P_{h,t},f)(M-C_6\varepsilon)
\end{eqnarray*}
with $C_6:= 2b(n,k)C_5$. The constants $M>0$ and $C_6$ are independent of $\varepsilon$. Hence, we can choose $\varepsilon>0$ so small that $M-C_6\varepsilon>0$.

By (\ref{35a}) we have 
$$W_k(P_{h,t},f)=\omega_{n-k}\int_{\Sn}f\,\D \Psi_k(P_{h,t},\cdot).$$ 
From $\lim_{t\to 0}P_{h,t}=K_h$ and the weak continuity of the $k$th area measure we deduce that
$$ \lim_{t\to 0} W_k(P_{h,t},f)= \omega_{n-k}\int_{\Sn} f\,\D\Psi_k(K_h,\cdot).$$
Since $\int f\,\D\Psi_k(K_h,\cdot)>0$ (which follows from \cite[(4.20), (4.26)]{Sch14} and the fact that all principal radii of curvature of the paraboloid ${\mathcal R}$ are positive), there is a positive number $W$ such that $b(n,k)^{-1}W_k(P_{h,t},f)\ge W$ for all sufficiently small $t>0$. We assume that $t$ is sufficiently small in this sense. Then we have
$$ |\Gamma(P_{h,t},f)(\vartheta E)-\Gamma(P_{h,t},f)(E)|\ge W(M-C_6\varepsilon)>0.$$

From the convergence $P_{h,t}\to K_h$ for $t\to 0$ and from the assumed weak continuity of $\Gamma$ it follows that
$$ |\Gamma(K_h,f)(\vartheta E) - \Gamma(K_h,f)(E)| \ge W(M-C_6\varepsilon)>0.$$
This completes the proof of Lemma 5.2 in the case where $2k\not= n-1$.

It remains to consider the case where
\begin{equation}\label{44}
2k=n-1\quad\mbox{and} \quad d\mbox{ is odd}.
\end{equation}
In this case, $n\ge 5$. As before, we denote by $\R^{n-1}$ the space spanned by the basis vectors $e_1,\dots,e_{n-1}$, and in addition by $\R^{n-2}$ the space spanned by $e_1,\dots,e_{n-2}$. In $\R^{n-1}$ we construct $(n-1)$-dimensional polytopes $P_{h,t}$ and convex bodies $K_h$ as above, just by replacing the previous triple $(\R^n,\R^{n-1},e_n)$ by 
$(\R^{n-1},\R^{n-2},e_{n-1})$. The numbers $t,h,\varepsilon$ and the set $\omega_h$ have the same meaning as before. The set $\omega_h$ is now a subset of the unit sphere of $\R^{n-1}$. We define the set
$$ \Omega_h:= \left\{ \sqrt{1-\alpha^2}\,v +\alpha e_n: v\in \omega_h,\,|\alpha|\le\varepsilon\right\}.$$
This set is invariant under the group ${\rm O}(n,e_n,e_{n-1})$, consisting of the rotations of $\R^n$ that fix $e_n$ and $e_{n-1}$, and it has nonempty interior in $\Sn$. Since $\omega_h$ is $\varepsilon$-close to $e_{n-1}$, we can, in view of (\ref{34a}) and (\ref{35}), choose $\varepsilon>0$ so small that
\begin{equation}\label{47}
\langle u,-e_{n-1}\rangle\ge 1-\varepsilon \quad \mbox{for all }u\in\Omega_h
\end{equation}
and
\begin{equation}\label{48}
|\langle u,a\rangle| \le\varepsilon \|a\| \quad\mbox{for all }u\in\Omega_h\mbox{ and all }a\in{\mathbb R}^{n-2}.
\end{equation}
Let $f$ be a nonnegative, continuous function on $\Sn$ that is invariant under ${\rm O}(n,e_n,e_{n-1})$, has its support in $\Omega_h$ and is not identically zero. We define
$$ {\mathcal F}_{k,h}(P_{h,t})=\{F\in{\mathcal F}_k(P_{h,t}):\Omega_h\cap \nu(P_{h,t},F)\not=\emptyset\}.$$
As above, we can choose $t_0>0$ so small that for $0<t<t_0$ each face of ${\mathcal F}_{k,h}(P_{h,t})$ lies above some face $F^\Box$ of the complex ${\mathcal C}_t$ in $\R^{n-2}$. The subspaces $L_1,\dots,L_{b(n-1,k)}\subset\R^{n-2}$ and the function $W_k(P_{h,t},f)$ are defined as before. Similar estimates as above show that for each $p$-tuple
$$ E:= (E',-e_{n-1},\dots,-e_{n-1}),\quad E':=(\underbrace{a,\dots,a}_{2q})$$
with $a\in \R^{n-2}$, $\|a\|=1$, an estimate corresponding to (\ref{53}) is valid, namely
\begin{equation}\label{49} 
\Gamma(P_{h,t},f)(E) = b(n-1,k)^{-1}W_k(P_{h,t},f)\left(\sum_{j=2}^d c_j Q^{q-j}\sum_{i=1}^{b(n-1,k)} Q_{L_i}^j\right)(E')+R_6(E)
\end{equation}
with 
$$ |R_6(E)|\le C_7W_k(P_{h,t},f)\varepsilon.$$

The tensor
$$ \Upsilon':= \sum_{j=2}^d c_j Q^{q-j}\sum_{i=1}^{b(n-1,k)} Q_{L_i}^j $$
now corresponds to the polynomial
$$ p_{\Upsilon'}(x)= \Upsilon'(\underbrace{x,\dots,x}_{2q}) =\sum_{j=2}^d c_j  \|x\|^{2(q-j)} \sum_{I\subset\{1,\dots,n-2\},\,|I|=k}\left(\sum_{i\in I} x_i^2\right)^j.$$
For
$$ x(\lambda,\mu)= \left(\lambda,\mu\sqrt{1-\lambda^2},\sqrt{1-\mu^2}\sqrt{1-\lambda^2},0,\dots,0\right)\in\R^n,\qquad \lambda,\mu\in[0,1],$$
we obtain (observing that $d$ is odd)
\begin{eqnarray*}
p_{\Upsilon'}(x(\lambda,\mu)) &=& dc_d\left[\binom{n-5}{k-2}-\binom{n-5}{k-1}\right]\mu^{2d-2}\lambda^{2d}+\dots,
\end{eqnarray*}
where the dots stand for terms in which the exponent of $\lambda$ is less than $2d$ or the exponent of $\mu$ is less than $2d-2$. Since $\binom{n-5}{k-2}-\binom{n-5}{k-1}\not=0$ for $2k=n-1$, this function is not constant.

The proof can now be completed as before. The essential fact that $\int f\,\D\Psi_k(K_h,\cdot)>0$, although now $K_h$ is of dimension $n-1$, remains true since $k\le n-2$.
\qed 

\vspace{2mm}

Finally, we consider the case $n=3$, $k=1$. We shall need Lemma \ref{Lemma5.1} for this case without change, hence all assumptions about $\Gamma, \varepsilon,\omega,s_0,q, E',E,f,c_j,d$ are as in that lemma. 

For even numbers $d$, the distinction between dimension three and higher dimensions in the proof of Lemma \ref{Lemma5.2} was not necessary, hence we can now assume that $d$ is odd. For odd numbers $d$, the previous proof breaks down, since for special values of $c_2,\dots,c_d$ with $c_d\not=0$ it can happen that the tensor $\Upsilon$ is invariant under ${\rm O}(3,e_3)$. Therefore, we need to modify the approximating polytopes $P_{h,t}$. We do this as follows.

Recall that we have identified $\R^2$ with the plane spanned by $e_1$ and $e_2$. Let $\beta_d:=\pi/d$. In $\R^2$, we define the vectors
$$ b_1 = e_1,\quad b_2 = (\cos\beta_d)e_1 +(\sin\beta_d)e_2\quad b_3 = b_2-b_1. $$
The triangle $T$ with vertices $0$, $b_1$ and $b_2$ has angles $\beta_d$ at $0$ and $((d-1)/2)\beta_d$ at $b_1$ and at $b_2$.
For $t>0$, the lines
$$ \R b_1+\zeta t b_2,\quad \R b_2+\zeta t b_3,\quad \R b_3+\zeta t b_1,\quad \zeta\in{\mathbb Z},$$
divide the plane $\R^2$ into triangles congruent to $T$. Together with their edges and vertices they define a polygonal complex in $\R^2$, which we denote by ${\mathcal T}_t$. In the construction of the polytopes $P_{h,t}$, as described after the proof of Lemma \ref{Lemma5.1}, we replace (for $n=3$) the complex ${\mathcal C}_t$ by the complex ${\mathcal T}_t$. Thus, we use the lifting map $\ell$, defined by (\ref{lift}), to lift the vertices of ${\mathcal T}_t$ to the paraboloid ${\mathcal R}$. The polyhedral set $R_t$ is the convex hull of the images. For $z\in\R^2$, the affine map $\alpha_z$ is again defined by (\ref{aff}), where now $r_t$ is the radius of the cirumcircle (that is, the circle through the vertices) of $T$. The closed disc bounded by the circumcircle  of $T$ contains no other vertices of ${\mathcal T}_t$ besides the vertices of T. By symmetry, the corresponding statement is true for all triangles of ${\mathcal T}_t$. Therefore, under orthogonal projection to $\R^2$, the faces of $R_t$ correspond precisely to the faces of ${\mathcal T}_t$. The role of the cubes $C_z$ of ${\mathcal C}_t$ is now played by the triangles $T_z$ of ${\mathcal T}_t$, where $z$ is the circumcenter (the center of the circumcircle) of $T_z$. On the vertices of $T_z$, the lifting map $\ell$ coincides with the affine map $\alpha_z$. Each face $F$ of $R_t$ is of the form $F=\alpha_z T_z$ for some triangle $T_z$ of ${\mathcal T}_t$.

In $\R^2$ we define the one-dimensional subspaces 
$$ L_r:= \R((\cos r\beta_d)e_1 +(\sin r\beta_d)e_2),\qquad r=0,1,\dots, d-1.$$
We observe that each edge of the complex ${\mathcal T}_t$ is parallel to one of the lines $L_0, L_1, L_{(d+1)/2}$; note that $(d+1)/2$ is an integer $\le d-1$.

The definitions of $K_h,P_{h,t},\omega_h,\lambda_z$ and the choices of $h,t_0$ are now, {\em mutatis mutandis}, the same as before. We assume that $0<t<t_0$. Let $\vartheta_d\in{\rm SO}(3,e_3)$ be the rotation by the angle $\beta_d$. The crucial new polytopes are the Minkowski averages
$$ P_{h,t}^d:= \frac{1}{d}\sum_{l=0}^{d-1}\vartheta_d^lP_{h,t}.$$
These polytopes clearly satisfy
$$ \vartheta_dP_{h,t}^d = P_{h,t}^d$$
and
$$ \lim_{t\to 0}P_{h,t}^d=K_h.$$
Define
$$ {\mathcal F}_{1,h}(P_{h,t}^d) := \{F\in{\mathcal F}_1(P_{h,t}^d):\omega_h\cap \nu(P_{h,t}^d,F)\not=\emptyset\}.$$
Let $F\in{\mathcal F}_{1,h}(P_{h,t}^d)$. Since $\omega_h$ is open and because of relation (2.26) in \cite{Sch14}, there exists a vector $u\in\omega_h\cap \nu(P_{h,t}^d,F)$ such that $F(P_{h,t}^d,u)=F$, where $F(K,u)$, for a convex body $K$, denotes the support set of $K$ with outer normal vector $u$. By \cite[Thm. 1.7.5(c)]{Sch14} we have
$$ F= \frac{1}{d}\sum_{l=0}^{d-1}F(\vartheta_d^lP_{h,t},u).$$
Since $\dim F=1$, for each $l\in\{0,\dots,d-1\}$ the support set $F(\vartheta_d^lP_{h,t},u)$ is either a vertex or an edge of $\vartheta_d^lP_{h,t}$ parallel to $F$, and in the latter case its orthogonal projection to $\R^2$ is parallel to one of the one-dimensional subspaces $L_r$, $r\in\{0,\dots,d-1\}$; here $r$ is independent of $l$, since all the considered edges are parallel. It follows that the orthogonal projection of the edge $F$ to $\R^2$ is parallel to $L_r$. We denote by ${\mathcal F}_{1,h}^{\, r}(P_{h,t}^d)$ the set of edges $F\in {\mathcal F}_{1,h}(P_{h,t}^d)$ for which the projection to $\R^2$ is parallel to $L_r$, $r=0,\dots,d-1$.

Let $f$ be a nonnegative, continuous function on ${\mathbb S}^2$ that is invariant under ${\rm O}(3,e_3)$, has its support in $\omega_h$ and is not identically zero.  

The fact that each edge of $P_{h,t}^d$ is parallel to an edge of some $\vartheta_d^lP_{h,t}$ has the consequence
that the former estimates, which depend only on the directions of such edges, remain valid. In particular, in the same way as (\ref{53}) was proved, we obtain 
\begin{eqnarray}\label{5.1} 
& & \Gamma(P_{h,t}^d,f)(E)\\
& & = \sum_{j=2}^d c_j  \sum_{r=0}^{d-1}  (Q^{q-j}Q_{L_r}^j)(E')\sum_{F\in{\mathcal F}^{\, r}_{1,h}(P_{h,t}^d)}
\Ha^1(F)\int_{\nu(P_{h,t}^d,F)} f\,\D{\mathcal H}^1 +R_7(E) \nonumber
\end{eqnarray}
with 
$$ |R_7(E)|\le C_8 W_1(P_{h,t}^d,f)\varepsilon.$$ 
Here
$$ \sum_{F\in{\mathcal F}^{\, r}_{1,h}(P_{h,t}^d)} \Ha^1(F)\int_{\nu(P_{h,t}^d,F)} f\,\D{\mathcal H}^1$$
is independent of $r$, since $P_{h,t}^d$ is invariant under $\vartheta_d$. The sum of these terms over all $r$ is equal to $W_1(P_{h,t}^d,f)$. Thus, we obtain
\begin{equation}\label{5.2} 
\Gamma(P_{h,t}^d,f)(E)=\frac{1}{d}W_1(P_{h,t}^d,f)\left(\sum_{j=2}^d c_j Q^{q-j}\sum_{r=0}^{d-1} Q_{L_r}^j\right)(E')+R_7(E).
\end{equation}

We want to show that the tensor
$$ \Upsilon_d:= \sum_{j=2}^d c_jQ^{q-j}\sum_{r=0}^{d-1}Q^j_{L_r}$$
is not invariant under ${\rm SO}(3,e_3)$. For $x=(x_1,x_2,0)\in\R^3$  with $\|x\|=1$ we have
$$ p_{\Upsilon_d}(x):=\Upsilon_d(\underbrace{x,\dots,x}_{2q}) =\sum_{j=2}^dc_j\sum_{r=0}^{d-1}(x_1\cos r\beta_d+x_2\sin r\beta_d)^{2j}.$$
For $x(\lambda):=(\lambda,\sqrt{1-\lambda^2},0)$, $\lambda\in[0,1]$, this reads
$$ p_{\Upsilon_d}(x(\lambda))=\sum_{j=2}^d c_j \sum_{r=0}^{d-1}\left(\lambda \cos r\beta_d + \sqrt{1-\lambda^2}\,\sin r \beta_d \right)^{2j}.$$
The right-hand side is a polynomial in $\lambda$, defined for all real $\lambda$, and for the coefficient $A_{2d}$ of $\lambda^{2d}$ in this polynomial we obtain
\begin{eqnarray*}
A_{2d} &=& \lim_{\lambda\to\infty}\lambda^{-2d}\sum_{j=2}^d c_j\sum_{r=0}^{d-1}\left(\lambda\cos r\beta_d + \sqrt{1-\lambda^2}\,\sin r\beta_d \right)^{2j}\\
&=& c_d \sum_{r=0}^{d-1}(\cos r\beta_d+{\rm i}\sin r\beta_d)^{2d}\\
&=& c_d \sum_{r=0}^{d-1} \exp\left(r\frac{\pi}{d}{\rm i}\cdot 2d\right) =dc_d \not=0.
\end{eqnarray*}
Hence, $p_{\Upsilon_d}(x(\lambda))$ is not a constant function of $\lambda$. Once we know this, the proof can be completed precisely as before.

\noindent Authors' addresses:\\[2mm]
Daniel Hug\\
Karlsruhe Institute of Technology, Department of Mathematics\\
D-76128 Karlsruhe, Germany\\
E-mail: daniel.hug@kit.edu\\[3mm]
Rolf Schneider\\
Mathematisches Institut, Albert-Ludwigs-Universit{\"a}t\\
D-79104 Freiburg i. Br., Germany\\
E-mail: rolf.schneider@math.uni-freiburg.de

\end{document}